\documentclass[11pt, a4paper]{article}
\usepackage{indentfirst}
\usepackage{amsfonts}
\usepackage{amssymb}
\usepackage{mathrsfs}
\usepackage{amsmath}
\usepackage{amsthm}
\usepackage{enumerate}
\usepackage{cite}
\usepackage{mathrsfs}
\allowdisplaybreaks
\usepackage{geometry}
\usepackage{ifpdf}
\ifpdf
\usepackage[colorlinks=true, linkcolor=blue, citecolor=red, final, backref=page, hyperindex]{hyperref}
\else
\usepackage[colorlinks, final, backref=page, hyperindex, hypertex]{hyperref}
\fi

\usepackage{txfonts}
\usepackage{indentfirst}
\usepackage{latexsym}
\usepackage[all]{xy}
\def\<{\langle}
\def\>{\rangle}

\newtheorem{thm}{Theorem}[section]
\newtheorem{lem}[thm]{Lemma}

\newtheorem{prop}[thm]{Proposition}
\newtheorem{ex}[thm]{Example}

\theoremstyle{definition}
\newtheorem{defn}{Definition}[section]

\theoremstyle{remark}
\newtheorem{re}{Remark}[section]
\begin{document}
	\title{\bf Maurer-Cartan characterization, cohomology and deformation of compatible Lie Yamaguti algebras}
	\author{\bf Asif Sania, Basdouri Imed, Sadraoui Mohamed Amin}
	\author{{ Asif Sania $^{1}$
			\footnote { Corresponding author, E-mail: 11835037@zju.edu.cn}
			Basdouri Imed $^{2}$
			\footnote { Corresponding author, E-mail: basdourimed@yahoo.fr}
			, \ Sadraoui Mohamed Amin $^{3}$
			\footnote { Corresponding author, E-mail: aminsadrawi@gmail.com}
		}\\
		{\small 1. Nanjing University of Information science and technology, Nanjing, PR China. } \\
		{\small 2. University of Gafsa, Faculty of Sciences Gafsa, 2112 Gafsa, Tunisia } \\
		{\small 3. University of Sfax, Faculty of Sciences Sfax, BP
			1171, 3038 Sfax, Tunisia}\\
	}
	\date{}
	\maketitle
	\begin{abstract}
		This study aims to generalize the notion of compatible Lie algebras to the compatible Lie Yamaguti algebras. Along with describing the representation of the compatible Lie Yamaguti algebra in detail, we also introduce the Maurer-Cartan characterization and cohomology of Lie Yamaguti algebras. As a result of the obtained cohomology, we studied its deformation. We define Rota-Baxter operators on compatible Lie Yamaguti algebras as well as on compatible pre-Lie Yamaguti algebras. Using Rota-Baxter operators, we examine how compatible Lie (and compatible pre-Lie) Yamaguti algebras are related.
	\end{abstract}
	
	
	\numberwithin{equation}{section}
	\tableofcontents
	\section{Introduction}\label{subsec11}
	In recent years, the concept of Lie Yamaguti algebra, which is a generalization of Lie algebra and Lie triple system (also denoted by L.t.s), has attracted the attention of numerous researchers to study and explore its advanced structural properties. Motivated by the study of Nomizu's in 1950, the structure of Lie Yamaguti algebra was first developed by K.Yamaguti in 1957, in the study of homogenous spaces with the canonical connection see \cite{N1, Y1}. Then K. Yamaguti continued this study in \cite{Y2, Y3}, which provided additional motivation to the other researchers and they started exploring this concept from different perspectives. This approach leads to the development of cohomology, deformation, and extension results of Lie Yamaguti algebra along with the structural development of pre-Lie Yamaguti algebra and Lie Yamaguti bialgebra, ( see \cite{B1, B2, B3, B4, L2 }for more details). In addition, the Lie Yamaguti algebra's enveloping algebra was examined in \cite{K1}, while Zhan and Li investigated the deformation and extension in \cite{Z1}. Note that a Lie Yamaguti algebra $(L, [\cdot, \cdot], [\cdot, \cdot, \cdot])$ with a trivial bilinear operation reduces to a Lie triple system, whereas a trivial trilinear operation reduces it to a Lie algebra. The Lie Yamaguti algebra \cite{Z1} also arises from a Lie algebra $(L, [\cdot, \cdot])$ endowed with a trilinear operation such that $[x, y, z]:= [[x, y], z]$ for $x, y, z \in L$ (for more explanation see example \eqref{ex2.1}).
 \par
	The concept of Rota-Baxter operators on associative algebra was first introduced by G. Baxter in his studies of fluctuation theory in probability \cite{B5}. In the study of the classical Yang-Baxter equation (C.Y.B.E), Kupershmidt proposed the idea of relative Rota-Baxter operators, also known as $\mathcal{O}$-operators or Kupershmidt operators, on Lie algebra. This broad concept has numerous significant implications for both physics and mathematical physics, as discussed in \cite{K2}. Let us briefly recall the definition of relative Rota-Baxter operators from \cite{L3}.\\
	Given the Lie algebra $(L, [\cdot, \cdot])$ and its representation $(V;\rho)$, a relative Rota-Baxter operator on $L$ with respect to $(V;\rho)$ is a linear map $R:V \rightarrow L$, that satisfies
	\begin{equation} \label{01}
		[Rx, Ry]=R\big(\rho(Rx)y-\rho(Ry)x \big) \quad \forall~x, y\in L. \end{equation}
	Thus, a relative Rota-Baxter operator is linked to an arbitrary representation. In contrast, Rota-Baxter operators (of weight $\lambda=0$) are related to the adjoint representation, which is what we are interested in this paper.
	Consequently, the equation \eqref{01} in this instance becomes
	\begin{equation*}
		[Rx, Ry]=R\big([Rx, y]-[Ry, x] \big).
	\end{equation*}
	For additional information on Rota-Baxter algebras, including their various applications, such as renormalization in perturbative quantum field theory, splitting of operads, and connections to noncommutative symmetric functions and Hopf algebras see \cite{B6, B7, B8} and references therein.
	Rota-Baxter operator, in particular, performs the transformation from Lie algebra $(L, [\cdot, \cdot])$ to pre-Lie algebra $(L, \star)$ using the formula
	\begin{equation*}
		x \star y= [Rx, y].
	\end{equation*}
	Since Lie Yamaguti algebra is a generalization of a Lie algebra and L.t.s, Rota-Baxter operators also exhibit this property (for more information on Rota-Baxter operators in the context of L.t.s, see \cite{C1}). Let us recall the notion of relative Rota-Baxter operator for Lie Yamaguti algebras from \cite{Z2, S1}.
	Given a Lie Yamaguti algebra $(L, [\cdot, \cdot], [\cdot, \cdot, \cdot])$ and its representation $(V;\rho, \mu)$, a relative Rota-Baxter operator on $L$ with respect to $(V;\rho, \mu)$ is a linear map $R:V \rightarrow L$, that satisfies the following identities 
	\begin{align*}
		[Rx, Ry]&=R\big(\rho(Rx)y-\rho(Ry)x, \big) \\
		[Rx, Ry, Rz]&=R\big(D_{\rho, \mu}(Rx, Ry)z + \mu(Ry, Rz)x - \mu(Rx, Rz)y. \big)
	\end{align*}
	Where $D_{\rho, \mu}$ is explained in the Definition \eqref{defn2.3}. One of the utilities of the Rota-Baxter operator is that it leads Lie Yamaguti algebra to pre-Lie Yamaguti algebra as mentioned in the next diagram:\begin{center} \label{diag1}
		$\xymatrix{\textbf{\emph{Lie Yamaguti algebras}}&\underrightarrow{ R.B.O} &\textbf{\emph{pre-Lie Yamaguti algebras.}}}$
	\end{center}
	\par A pair of Lie algebras $(L, [\cdot, \cdot]_1)$ and $(L,[\cdot,\cdot]_2)$ that are compatible with one another still contain Lie brackets in any linear combination in which they appear. If the compatibility requirement is met between two Lie algebras, then the new algebra formed is called a compatible Lie algebra $(L, [\cdot, \cdot]_1, [\cdot, \cdot]_2)$, studied in \cite{G2, G3, O1, L1}. The notion of compatible Lie algebras appears in many fields of mathematics and mathematical physics such as the study of the classical Yang-Baxter equations, principal chiral field, loop algebras, etc. Due to the significant importance of compatible algebras structure in Cartan matrices of affine Dynkin diagrams, integrable matrix equations, Hamiltonian dynamical systems, infinitesimal bialgebras, and quiver representations, we are motivated to study the compatible Lie Yamaguti algebras and their relations with the famous Rota-Baxter operator. In this paper, we generalize the concept of compatible Lie algebra to the case of compatible Lie Yamaguti algebras and show its relation with the Rota-Baxter operators and pre-Lie Yamaguti algebras. We discuss their representation theory, which provides a foundation to evaluate the cohomology and deformation theory of compatible Lie Yamaguti algebra.
	\par This paper is organized as follows:
	Section \eqref{sec2} briefly recalls some basic definitions and results that we need in the rest of the paper. In Section \eqref{sec3}, we propose the idea of compatible Lie Yamaguti algebras and their representation. In addition to defining derivation for the compatible Lie algebra and providing an appropriate example, we explicitly explain the representation of the semi-direct product of the compatible Lie Yamaguti algebra. In Section \eqref{section cohomolgy} we introduce the cohomolgy of compatible Lie Yamaguti algebras using the Maurer-Cartan characterization. In Section \eqref{section deformation} we study the deformation of compatible Lie Yamaguti algebras and we show the difference between our study and the case of compatible Lie algebras.  Finally, in Section \eqref{sec4}, we define the Rota-Baxter operators on compatible Lie Yamaguti algebras and explore the compatible pre-Lie Yamaguti algebras. We further investigate the relationship between compatible Lie (and pre-Lie) compatible Lie Yamaguti algebras via Rota-Baxter operators.\\Throughout this paper, we consider all the vector spaces and tensor products over the field $\mathbb{K}$ of characteristic $0$ and Rota-Baxter operator of weight $\lambda =0$.
	\section{Preliminary} \label{sec2} This section is devoted to defining some important definitions and notions, that are useful for further understanding of this paper. However, this section is divided into two subsections, In the first subsection, we introduce Lie  Yamaguti algebra and its representations. In the later subsection, we recall the notion of the Rota-Baxter operator on the Lie Yamaguti algebras and provide some important results.
	\subsection*{Lie Yamaguti algebras:} Some basic results on Lie Yamaguti algebras, described in this subsection are as follows
	\begin{defn}
		A Lie Yamaguti algebra is a vector space $L$ equipped with both bilinear and trilinear brackets respectively 
		\begin{center}
			$[\cdot, \cdot]: \wedge^2L \rightarrow L$ and $[\cdot, \cdot, \cdot]: \wedge^2L \otimes L \rightarrow L$
		\end{center}
		that satisfies the following equations, for all $x, y, z, w, t \in L$
		\begin{enumerate}
			\item[(LY1):] $[[x, y], z]+[[y, z], x]+[[z, x], y]+[x, y, z]+[y, z, x]+[z, x, y]=0, $
			\item[(LY2):] $[[x, y], z, w]+[[y, z], x, w]+[[z, x], y, w]=0, $
			\item[(LY3):] $[x, y, [z, w]]=[[x, y, z], w]+[z, [x, y, w]], $
			\item[(LY4):]$ [x, y, [z, w, t]]=[[x, y, z], w, t]+[z, [x, y, w], t]+[z, w, [x, y, t]].$
		\end{enumerate}
	\end{defn}
	\begin{ex} \label{ex2.1}
		Let $(L, [\cdot, \cdot])$ be a Lie algebra. We define $[\cdot, \cdot, \cdot]: \wedge^2L \otimes L \rightarrow L$ by 
		\begin{align*}
		  [x, y, z]:=[[x, y], z], \quad\forall~ x, y, z \in L.  
		\end{align*}
		Then $(L, [\cdot, \cdot], [\cdot, \cdot, \cdot])$ naturally yields a Lie Yamaguti algebra.
	\end{ex}
	\begin{defn}
		Let $(L, [\cdot, \cdot]_L, [\cdot, \cdot, \cdot]_L)$ and $(H, [\cdot, \cdot]_H, [\cdot, \cdot, \cdot]_H)$ two Lie Yamaguti algebras. An homomorphism from $(L, [\cdot, \cdot]_L, [\cdot, \cdot, \cdot]_L)$ to $(H, [\cdot, \cdot]_H, [\cdot, \cdot, \cdot]_H)$ is a linear map $\varphi :L \rightarrow H$, such that for all $x, y, z \in L$ 
		\begin{align*}
			\varphi([x, y]_L)&=[\varphi(x), \varphi(y)]_H, \\
			\varphi([x, y, z]_L)&=[\varphi(x), \varphi(y), \varphi(z)]_H.
		\end{align*}
	\end{defn}
	Now we recall the definition of representation of a Lie Yamaguti algebra.
	\begin{defn} \label{defn2.3}
		Let $(L, [\cdot, \cdot], [\cdot, \cdot, \cdot])$ be a Lie Yamaguti algebra and $V$ a vector space. A representation of $L$ on $V$ consists of both a linear and a bilinear maps respectively $\rho : L \rightarrow gl(V)$ and $\mu : \otimes^2L \rightarrow gl(V)$ such that for all $x, y, z, w \in L$, we have
		\begin{enumerate}
			\item $\mu([x, y], z)-\mu(x, z)\rho(y)+\mu(y, z)\rho(x)=0, $
			\item $\mu(x, [y, z])-\rho(y)\mu(x, z)+\rho(z)\mu(x, y)=0, $
			\item $\rho([x, y, z])=[D_{\rho, \mu}(x, y), \rho(z)], $
			\item $\mu(z, w)\mu(x, y)-\mu(y, w)\mu(x, z)-\mu(x[y, z, w])+D_{\rho, \mu}(y, z)\mu(x, w)=0, $
			\item $\mu([x, y, z], w)+\mu(z, [x, y, w])=[D_{\rho, \mu}(x, y), \mu(z, w)], $
		\end{enumerate}
		where the bilinear map $D_{\rho, \mu}:\otimes^2L \rightarrow gl(V)$ is given by 
		\begin{equation*}
			D_{\rho, \mu}(x, y):=\mu(y, x)-\mu(x, y)+[\rho(x), \rho(y)]-\rho([x, y]) \quad \forall x, y \in L.\end{equation*}
		We denote  the representation $L$ on $V$ by $(V;\rho, \mu)$.
	\end{defn}
	\begin{re}
		Let $(L, [\cdot, \cdot], [\cdot, \cdot, \cdot])$ be a Lie Yamaguti algebra and $(V;\rho, \mu)$ be its representation. The Lie Yamaguti algebra reduces to a L.t.s $(L, [\cdot, \cdot, \cdot])$ if $\rho=0$. In this case $(V;\mu)$ is a representation of L.t.s $(L, [\cdot, \cdot, \cdot])$. Also, the Lie Yamaguti algebra $L$ reduces to a Lie algebra $(L, [\cdot, \cdot])$ if $\mu=0$. In this case $D_{\rho, \mu}=0$, which means $(V;\rho)$ becomes a representation of $(L, [\cdot, \cdot])$. As we mentioned earlier the Lie Yamaguti algebra is a generalization of both Lie algebra and L.t.s, the previous definition confirms this generalization.
	\end{re}
	Many authors characterized the representation of a Lie Yamaguti algebra by the semi-direct product of  Lie Yamaguti algebra with its representations, let's recall it in the next proposition.
	\begin{prop}
		Let $(L, [\cdot, \cdot], [\cdot, \cdot, \cdot])$ be a Lie Yamaguti algebra and $V$ a vector space. Let $\rho : L \rightarrow gl(V)$ and $\mu : \otimes^2L \rightarrow gl(V)$ are linear maps. Then $(V;\rho, \mu)$ is a representation of $(L, [\cdot, \cdot], [\cdot, \cdot, \cdot])$ if and only if there is a Lie Yamaguti algebra structure on the direct sum $L \oplus V$, which is defined by 
		\begin{align*}
			[x+u, y+v]_{\ltimes}&=[x, y]+\rho(x)v-\rho(y)u, \\
			[x+u, y+v, z+w]_{\ltimes}&=[x, y, z]+ D_{\rho, \mu}(x, y)w+\mu(y, z)u-\mu(x, z)v.
		\end{align*}for all $x, y, z \in L$; $u, v, w \in V$.
		Such a Lie Yamaguti algebra $(L \oplus V, [\cdot, \cdot]_{\ltimes}, [\cdot, \cdot, \cdot]_{\ltimes})$ is called the semidirect product Lie Yamaguti algebra. 
	\end{prop}
	\subsection*{Rota-Baxter operators on Lie Yamaguti algebras}
	In this subsection, we recall some results of Rota-Baxter operators on Lie Yamaguti algebras. For more details see \cite{Z2}, where this concept is elaborated to the general case.
	\begin{defn}
		Let $(L, [\cdot, \cdot], [\cdot, \cdot, \cdot])$ be a Lie Yamaguti algebra. A linear map $R: L \rightarrow L$ is called Rota-Baxter operator on $L$, if $R$ satisfies 
		\begin{align*}
			[Rx, Ry]&=R \big([Rx, y]-[Ry, x] \big), \\
			[Rx, Ry, Rz]&=R \big( [Rx, Ry, z]+[Ry, Rz, x]-[Rx, Rz, y]), 
		\end{align*}
		for all $x, y, z \in L$.
	\end{defn}
	Let's recall the definition of pre-Lie Yamaguti algebras \cite{S1}.
	\begin{defn}
		A pre-Lie Yamaguti algebra $(L, \star, \{\cdot, \cdot, \cdot\})$ is a vector space $L$ equipped with a bilinear operation $\star : \otimes^2 L \rightarrow L$ and a trilinear operation $\{\cdot, \cdot, \cdot\}: \otimes^3 L \rightarrow L$ such that for all $x, y, z, w, t \in L$, the following identities holds 
		\begin{enumerate} \item $\{z, [x, y]_C, w\}-\{y \star z, x, w\}+\{x \star z, y, w\}=0, $
			\item $\{x, y, [z, w]_C\}=z \star \{x, y, w\}-w \star \{x, y, z\}, $ 
			\item $\{\{x, y, z\}, w, t\}- \{\{x, y, w\}, z, t\}- \{x, y, \{z, w, t\}_D\}- \{x, y, \{z, w, t\}\} +\{x, y, \{w, z, t\}\}+\{z, w, \{x, y, t\}\}_D= 0, $ 
			\item $\{z, \{x, y, w\}_D, t\}+\{z, \{x, y, w\}, t\}-\{z, \{y, x, w\}, t\}+\{z, w, \{x, y, t\}_D\} 
			+\{z, w, \{x, y, t\}\}-\{z, w, \{y, x, t\}\}\\=\{x, y, \{z, w, t\}\}_D-\{\{x, y, z\}_D, w, t\}, $ 
			\item $\{x, y, z\}_D \star w+\{x, y, z\} \star w -\{y, x, z\} \star w=\{x, y, z\star w\}_D-z \star \{x, y, w\}_D.$
		\end{enumerate}With $[\cdot, \cdot]_C:\wedge^2L \rightarrow L$ is the commutator and $\{\cdot, \cdot, \cdot\}_D:\wedge^2L\otimes L \rightarrow L$ are respectively defined as follows
		\begin{align*}
			[x, y]_C&:=x \star y -y \star x, \\
			\{x, y, z\}_D&:=\{z, y, x\}-\{z, x, y\}+(y, x, z)-(x, y, z), 
		\end{align*}for $x, y, z \in L$. Here $(\cdot, \cdot, \cdot)$ means the associator which is defined by 
		\begin{center}
			$(x, y, z):=(x \star y)\star z -x \star (y \star z)$.
		\end{center}
	\end{defn}
	The diagram \eqref{diag1} holds by the next theorem 
	\begin{thm}(Theorem $(3.13)$ in \cite{S1}) Let $R: L \rightarrow L$ be a Rota-Baxter operator on a Lie Yamaguti algebra $(L, [\cdot, \cdot], [\cdot, \cdot, \cdot])$. Define two linear maps $\star:\otimes^2 L \to L $ and $ \{\cdot, \cdot, \cdot\}: \otimes^3 L \to L$ by
		\begin{align*}x\star y=[Rx, y] \text{ and } \{x, y, z\}=[Ry, Rz, x].
		\end{align*}
		Then $(L, \star, \{\cdot, \cdot, \cdot\})$ is a pre-Lie Yamaguti algebra.
	\end{thm}
	\section{ Compatible Lie Yamaguti algebras} \label{sec3}
	In this section, we introduce the notion of compatible Lie Yamaguti algebra. Based on the compatible Lie algebras in \cite{L1}, we generalize it to the case of Lie Yamaguti algebra. Furthermore, we introduce the main result of the paper. We start with the following definition.
	\begin{defn}
		\label{defn3.1}
		A compatible Lie Yamaguti algebra is a $5$-tuple $(L, [\cdot, \cdot]_1, [\cdot, \cdot, \cdot]_1, [\cdot, \cdot]_2, [\cdot, \cdot, \cdot]_2)$, where $(L, [\cdot, \cdot]_1, [\cdot, \cdot, \cdot]_1)$ and $(L, [\cdot, \cdot]_2, [\cdot, \cdot, \cdot]_2)$ are two Lie Yamaguti algebras such that 
		\begin{enumerate}
			\item[(\textbf{CY1}):] $[[x, y]_1, z]_2+c.p.+[[x, y]_2, z]_1+c.p.=0, $
			\item[(\textbf{CY2}):] $ [[x, y]_1, z, t]_2+[[y, z]_1, x, t]_2+[[z, x]_1, y, t]_2+[[x, y]_2, z, t]_1+[[y, z]_2, x, t]_1+[[z, x]_2, y, t]_1=0, $
			\item[(\textbf{CY3}):] $ [a, b, [x, y, z]_1]_2+[a, b, [x, y, z]_2]_1=[[a, b, x]_1, y, z]_2+[x, [a, b, y]_1, z]_2\\+[x, y, [a, b, z]_1]_2+ [[a, b, x]_2, y, z]_1+[x, [a, b, y]_2, z]_1+[x, y, [a, b, z]_2]_1, $
			\item[(\textbf{CY4}):] $[a, b, [x, y]_1]_2+[a, b, [x, y]_2]_1=[[a, b, x]_1, y]_2+[x, [a, b, y]_1]_2+[[a, b, x]_2, y]_1+[x, [a, b, y]_2]_1.$
		\end{enumerate}
		Where $c.p.$ means cyclic permutations concerning $x, y \text{ and } z$. For simplicity, we abbreviate the compatible Lie Yamaguti algebra by \textbf{C.L.Y.A}.
	\end{defn}
	\begin{ex}
		Consider a $2$-dimensional Lie Yamaguti algebra $L$ with the basis $\{e_1,e_2\}$, where brackets are defined by:
		$[e_1,e_2]_1=e_1,[e_1,e_2,e_2]_1=e_1,[e_1,e_2]_2=e_2 \text{ and } [e_1,e_2,e_2]_2=e_2$. In this case $(L,[\cdot,\cdot]_1,[\cdot,\cdot,\cdot]_1,[\cdot,\cdot]_2,[\cdot,\cdot,\cdot]_2)$ is a \textbf{C.L.Y.A}.
	\end{ex}
	\begin{ex}
		Let $(L,[\cdot,\cdot]_1,[\cdot,\cdot]_2)$ be a compatible Lie algebras. Define 
		\begin{equation*}
			[\cdot,\cdot,\cdot]_i:\wedge^2L \otimes L \rightarrow L \text{ by } [x,y,z]_i=[[x,y]_j,z]_i, \text{ for } x,y,z \in L \text{ and } i, j\in \{1,2\},i\neq j .
		\end{equation*}
		Then $(L,[\cdot,\cdot]_1,[\cdot,\cdot,\cdot]_1,[\cdot,\cdot]_2,[\cdot,\cdot,\cdot]_2)$ is a \textbf{C.L.Y.A}.
	\end{ex}
	\begin{prop}
		\label{prop3.1}
		A $5$-tuple $(L, [\cdot, \cdot]_1, [\cdot, \cdot, \cdot]_1, [\cdot, \cdot]_2, [\cdot, \cdot, \cdot]_2)$ is a compatible Lie Yamaguti algebras if and only if $(L,[\cdot, \cdot]_1, [\cdot, \cdot, \cdot]_1)$ and $(L,[\cdot, \cdot]_2, [\cdot, \cdot, \cdot]_2)$ are Lie Yamaguti structures on $L$ such that for any $k_1, k_2 \in \mathbb{K}$, the new bilinear and trilinear operation 
		\begin{align}
			\lceil \cdot, \cdot \rceil&=k_1 [\cdot, \cdot]_1+k_2 [\cdot, \cdot]_2 \label{eq1.1} \\
			\lceil \cdot, \cdot, \cdot \rceil&=k_1 [\cdot, \cdot, \cdot]_1+k_2 [\cdot, \cdot, \cdot]_2 \label{eq1.2}
		\end{align}
		define the Lie Yamaguti structure on $L$.
	\end{prop}
	Let $(L_1, [\cdot, \cdot]_1)$ and $(L_2, [\cdot, \cdot]_2)$ be two Lie algebras, then $(L_1 \oplus L_2, [\cdot, \cdot]_{\ltimes})$ is a Lie algebra with the bracket $[\cdot, \cdot]_{\ltimes}=[\cdot, \cdot]_1+[\cdot, \cdot]_2$. In the next proposition, we show that the direct sum of two Lie Yamaguti algebras is a Lie Yamaguti algebra.
	\begin{prop}
		Let $(L_1, [\cdot, \cdot]_1, [\cdot, \cdot, \cdot]_1)$ and $(L_2, [\cdot, \cdot]_2, [\cdot, \cdot, \cdot]_2)$ be two Lie Yamaguti algebras, there is a Lie Yamaguti algebra $(L_1 \oplus L_2, [\cdot, \cdot]_{\ltimes}, [\cdot, \cdot, \cdot]_{\ltimes})$ where 
		\begin{align*}
			{[x_1+x_2, y_1+y_2]}_{\ltimes}&={[x_1, y_1]}_1+{[x_2, y_2]}_2 \\
			{[x_1+x_2, y_1+y_2, z_1+z_2]}_{\ltimes}&={[x_1, y_1, z_1]}_1+{[x_2, y_2, z_2]}_2
		\end{align*}
	\end{prop}
	We generalize this proposition to the case of compatible Lie Yamaguti algebras as follows
	\begin{prop}
		Let $(L_1, \lceil \cdot, \cdot \rceil_1, \lceil \cdot, \cdot, \cdot \rceil_1)$ and $(L_2, \lceil \cdot, \cdot \rceil_2, \lceil \cdot, \cdot, \cdot \rceil_2)$ be two compatible Lie Yamaguti algebras, there is a compatible Lie Yamaguti algebras $(L_1 \oplus L_2, \lceil \cdot, \cdot \rceil_{\ltimes}, \lceil \cdot, \cdot, \cdot \rceil_{\ltimes})$ where 
		\begin{align*}
			\lceil \cdot, \cdot \rceil_{\ltimes}&=\lceil \cdot, \cdot \rceil_1+\lceil \cdot, \cdot \rceil_2 \\
			\lceil \cdot, \cdot, \cdot \rceil_{\ltimes}&=\lceil \cdot, \cdot, \cdot \rceil_1+\lceil \cdot, \cdot, \cdot \rceil_2.
		\end{align*}
	\end{prop}
	The definition of a homomorphism between two Lie Yamaguti algebras is provided in \cite{T1}, while that of a homomorphism between two compatible Lie algebras is provided in \cite{L1}. To define the homomorphism between two compatible Lie Yamaguti algebras, we combine the previous two definitions as follows.
	\begin{defn}
		A homomorphism between two compatible Lie Yamaguti algebras\\ $(L, [\cdot, \cdot]_1, [\cdot, \cdot, \cdot]_1, [\cdot, \cdot]_2, [\cdot, \cdot, \cdot]_2)$ and $(K, \{\cdot, \cdot\}_1, \{\cdot, \cdot, \cdot\}_1, \{\cdot, \cdot\}_2, \{\cdot, \cdot, \cdot\}_2)$ is a linear map $\varphi : L \rightarrow K$ such that $\varphi$ is a Lie-Yamaguati algebra homomorphism between $(L, [\cdot, \cdot]_1, [\cdot, \cdot, \cdot]_1)$ and $(K, \{\cdot, \cdot\}_1, \{\cdot, \cdot, \cdot\}_1)$, as well as between $(L, [\cdot, \cdot]_2, [\cdot, \cdot, \cdot]_2)$ and $(K, \{\cdot, \cdot\}_2, \{\cdot, \cdot, \cdot\}_2)$. In addition, for $i=1, 2$, we have
		\begin{align*}
			\varphi([x, y]_i)&=\{\varphi(x), \varphi(y)\}_i \\
			\varphi([x, y, z]_i)&=\{\varphi(x), \varphi(y), \varphi(z)\}_i.
		\end{align*}
	\end{defn}
	\begin{defn}
		\label{def3.3}
		A representation of a compatible Lie Yamaguti algebras $(L, \lceil\cdot, \cdot \rceil, \lceil\cdot, \cdot, \cdot \rceil )$ on a vector space $V$ is couple $(\rho, \mu)$ where 
		\begin{align*} &\lceil\cdot, \cdot \rceil=k_1[\cdot, \cdot]_1+k_2[\cdot, \cdot]_2, \\ &\lceil\cdot, \cdot, \cdot \rceil=k_1[\cdot, \cdot, \cdot]_1+k_2[\cdot, \cdot, \cdot]_2 , \\
			&\rho=k_1 \rho_1+k_2 \rho_2 \text{ and } \mu=k_1 \mu_1+k_2 \mu_2.
		\end{align*}
		such that:\begin{itemize}
		    \item  $(V;\rho_1, \mu_1)$ is a representation of $(L, [\cdot, \cdot]_1, [\cdot, \cdot, \cdot]_1)$ on $V$,
      \item $(V;\rho_2, \mu_2)$ is a representation of $(L, [\cdot, \cdot]_2, [\cdot, \cdot, \cdot]_2)$ on $V$, 
      \item 
      and the following equations hold:
		\begin{enumerate}\item $ \mu_2([x, y]_1, z)+\mu_1([x, y]_2, z)= \mu_2(x, z)\rho_1(y)+\mu_1(x, z)\rho_2(y)-\mu_2(y, z)\rho_1(x)-\mu_1(y, z)\rho_2(x), $ 
			\item$ \mu_2(x, [y, z]_1)+\mu_1(x, [y, z]_2)=\rho_2(y)\mu_1(x, z)+\rho_1(y)\mu_2(x, z)-\rho_2(z)\mu_1(x, y)-\rho_1(z)\mu_2(x, y) , $ 
			\item$ \rho_1([x, y, z]_2)+\rho_2([x, y, z]_1)=[D_1(x, y), \rho_2(z)]+[D_2(x, y), \rho_1(z)], $  \item $\mu_2(z, w)\mu_1(x, y)+\mu_1(z, w)\mu_2(x, y)+D_2(y, z)\mu_1(x, w)+D_1(y, z)\mu_2(x, w)\\=\mu_2(y, w)\mu_1(x, z)+\mu_1(y, w)\mu_2(x, z) +\mu_2(x, [y, z, w]_1)+\mu_1(x, [y, z, w]_2) $, 
			\item $\mu_2([x, y, z]_1, w)+\mu_1([x, y, z]_2, w)+\mu_2(z, [x, y, w]_1)+\mu_1(z, [x, y, w]_2)\\=[D_2(x, y), \mu_1(z, w)]_1+[D_1(x, y), \mu_2(z, w)]_2, $\item
			$\rho_1([x, y]_2)+\rho_2([x, y]_1)=\rho_1(x)\rho_2(y)-\rho_1(y)\rho_2(x)
			+\rho_2(x)\rho_1(y)-\rho_2(y)\rho_1(x). \label{eq1.3}$\end{enumerate}
		We denote $D=D_{\rho, \mu}$, $D_1=D_{\rho_1, \mu_1}$ and $D_2=D_{\rho_2, \mu_2}$. From Eq. \eqref{eq1.3}, we obtain 
		$D=D_1+ D_2$.\end{itemize} 
	\end{defn}
	\begin{ex}
		Let $(L, [\cdot, \cdot]_1, [\cdot, \cdot, \cdot]_1, [\cdot, \cdot]_2, [\cdot, \cdot, \cdot]_2)$ be a compatible Lie Yamaguti algebras. We define linear maps $ad:L \rightarrow gl(L)$ and $\mathfrak{ad}:\otimes^2L \rightarrow gl(L)$ by $x \mapsto ad_x$ and $(x, y) \mapsto \mathfrak{ad}_{x, y}$ respectively. More specifically \begin{align*}
		    ad_x(z)=\lceil x, z\rceil \text{ and } \mathfrak{ad}_{x, y}(z)=\lceil z, x, y\rceil ,~~~\forall~~ z \in L.
		\end{align*}
		Then $(ad, \mathfrak{ad})$ forms a representation of $L$ on itself, where $\mathfrak{L}:=D_{ad, \mathfrak{ad}}$ is given by 
		\begin{equation*}
			\mathfrak{L}_{x, y}(z)=[x, y, z], \quad \forall z \in L.
		\end{equation*}
		Where 
		\begin{align*}
			ad&=k_1 ad_1 + k_2 ad_2, \\
			\mathfrak{ad}&=k_1 \mathfrak{ad}_1 + k_2 \mathfrak{ad}_2, \\
			\mathfrak{L}&=\mathfrak{L}_1+\mathfrak{L}_2;\quad (\mathfrak{L}_i=D_{ad_i, \mathfrak{ad}_i}, i =1, 2), \\
			\lceil \cdot, \cdot \rceil&=k_1 [\cdot, \cdot]_1 + k_2 [\cdot, \cdot]_2, \\
			\lceil \cdot, \cdot, \cdot \rceil&=k_1 [\cdot, \cdot, \cdot]_1+k_2 [\cdot, \cdot, \cdot]_2.
		\end{align*}
	\end{ex}
	The representation of a compatible Lie Yamaguti algebra can be characterized by $ \textbf{\emph{ the semidirect product}}$ compatible Lie Yamaguti algebras, like the case of Lie algebras. In \cite{Z3} the authors in proposition 2.10, have defined $ \textbf{\emph{ the semidirect product }}$ Lie Yamaguti algebras. In the next proposition, we define $\textbf{\emph{the semidirect product }}$ of compatible Lie Yamaguti algebras
	\begin{prop}
		Let $(V;\rho, \mu)$ be a representation of a compatible Lie Yamaguti algebras $(L, \lceil \cdot, \cdot \rceil, \lceil \cdot, \cdot, \cdot \rceil)$. Define two skew-symmetric bilinear operation $\{ \cdot, \cdot\}_1, \{ \cdot, \cdot\}_2 : \wedge^2(L \oplus V) \rightarrow L \oplus V$ by 
		\begin{align*}
			\{x+u, y+v\}_1&=[x, y]_1+\rho_1(x)v-\rho_1(y)u, \\
			\{x+u, y+v\}_2&=[x, y]_2+\rho_2(x)v-\rho_2(y)u.
		\end{align*}
		And define two trilinear operations $\{ \cdot, \cdot, \cdot\}_1, \{ \cdot, \cdot, \cdot\}_2 : \wedge^2(L \oplus V) \otimes (L\oplus V) \rightarrow L \oplus V$ by 
		\begin{align*}
			\{x+u, y+v, z+w\}_1&=[x, y, z]_1+\mu_1(y, z)u-\mu_1(x, z)v+D_1(x, y)w, \\
			\{x+u, y+v, z+w\}_2&=[x, y, z]_2+\mu_2(y, z)u-\mu_2(x, z)v+D_2(x, y)w,
		\end{align*}
		for any $x, y \in L$ and $u, v, w \in V$. Then $(L \oplus V, \{ \cdot, \cdot\}_1, \{ \cdot, \cdot, \cdot\}_1, \{ \cdot, \cdot\}_2, \{ \cdot, \cdot, \cdot\}_2)$ \Big(or simply denoted by $(L \oplus V, \{ \cdot, \cdot\}_C, \{ \cdot, \cdot, \cdot\}_C )$
		where $\{ \cdot, \cdot\}_C=k_1 \{ \cdot, \cdot\}_1+k_2 \{ \cdot, \cdot\}_2$ and $\{ \cdot, \cdot, \cdot\}_C=k_1 \{ \cdot, \cdot, \cdot\}_1 + k_2 \{ \cdot, \cdot, \cdot\}_2$ \Big) is a compatible Lie Yamaguti algebras with 
		\begin{align*}
			\rho&=k_1 \rho_1 +k_2\rho_2, 
			\qquad\mu=k_1 \mu_1 +k_2\mu_2, \\
			D_1(x, y)&= \mu_1(y, x)-\mu_1(x, y), 
			\qquad D_2(x, y)= \mu_2(y, x)-\mu_2(x, y), \\
			\lceil \cdot, \cdot\rceil&=k_1 [\cdot, \cdot]_1+k_2 [\cdot, \cdot]_2, \qquad 
			\lceil \cdot, \cdot, \cdot \rceil=k_1 [\cdot, \cdot, \cdot]_1+k_2 [\cdot, \cdot, \cdot]_2.
		\end{align*} 
	\end{prop}
	\begin{proof} For $x, y, z \in L$ and $u, v, w \in V$, we have 
		\begin{align*}
			\{x+u, y+v\}_C&=k_1 \{x+u, y+v\}_1 + k_2 \{x+u, y+v\}_2 \\
			&=k_1 \Big([x, y]_1+\rho_1(x)v-\rho_1(y)u \Big)+k_2 \Big([x, y]_2+\rho_2(x)v-\rho_2(y)u \Big) \\
			&=\lceil x, y \rceil + \rho(x)v-\rho(y)u \\
			&=-(\lceil y, x \rceil +\rho(y)u -\rho(x)v) \\
			&=-\{y+v, x+u\}_C
		\end{align*}
		that gives the skew symmetry of $\{ \cdot, \cdot\}_C$, similarly for $\{ \cdot, \cdot, \cdot\}_C$. \\
		For the first identity $(LY1),$ we have 
		\begin{align*}
			&\{\{x+u, y+v\}_C, z+w\}_C\\=&\{k_1\{x+u, y+v\}_1+k_2\{x+u, y+v\}_2, z+w\}_C \\
			=&k_1\{\{x+u, y+v\}_1, z+w\}_C+k_2 \{\{x+u, y+v\}_2, z+w\}_C \\
			=&k_1 \Big(k_1 \{\{x+u, y+v\}_1, z+w\}_1+k_2 \{\{x+u, y+v\}_1, z+w\}_2 \Big) \\
			&+k_2 \Big(k_1 \{\{x+u, y+v\}_2, z+w\}_1+k_2 \{\{x+u, y+v\}_2, z+w\}_2 \Big) \\
			=&k_1^2\{[x, y]_1+\rho_1(x)v-\rho_1(y)u, z+w\}_1 + k_1k_2 \{[x, y]_1+\rho_1(x)v-\rho_1(y)u, z+w\}_2 \\
			&+k_2k_1 \{[x, y]_2+\rho_2(x)v-\rho_2(y)u, z+w\}_1+k_2^2 \{[x, y]_2+\rho_2(x)v-\rho_2(y)u, z+w\}_2 \\
			=&k_1^2 \Big([[x, y]_1, z]_1+\rho_1([x, y]_1)w-\rho_1(z)W_1 \Big)+k_1k_2 \Big([[x, y]_1, z]_2+\rho_2([x, y]_1)w-\rho_2(z)W_1 \Big) \\
			&+k_2k_1 \Big([[x, y]_2, z]_1+\rho_1([x, y]_2)w-\rho_1(z)W_2 \Big) + k_2^2 \Big([[x, y]_2, z]_2+\rho_2([x, y]_2)w-\rho_2(z)W_2 \Big)
		\end{align*}
		where $W_1= \rho_1(x)v- \rho_1(y)u$ and $W_2= \rho_2(x)v- \rho_2(y)u$.\\
		Similarly, we obtain 
		\begin{align*}
			&\{\{y+v, z+w\}_C, x+u\}_C\\=&
			k_1^2 \Big([[y, z]_1, x]_1+\rho_1([y, z]_1)u-\rho_1(x)U_1 \Big) 
			+k_1k_2 \Big([[y, z]_1, x]_2+\rho_2([y, z]_1)u-\rho_2(x)U_1 \Big) \\
			&+k_2k_1 \Big([[y, z]_2, x]_1+\rho_1([y, z]_2)u-\rho_1(x)U_2 \Big) 
			+k_2^2 \Big([[y, z]_2, x]_2+\rho_2([y, z]_2)u-\rho_2(x)U_2 \Big) 
		\end{align*}
		where $U_1=\rho_1(y)w-\rho_1(z)v$ and $U_2=\rho_2(y)w-\rho_2(z)v$.\\
		And finally, we obtain 
		\begin{align*}
			&\{\{z+w, x+u\}_C, y+v\}_C\\=&
			k_1^2 \Big([[z, x]_1, y]_1+\rho_1([z, x]_1)v-\rho_1(y)V_1 \Big) 
			+k_1k_2 \Big([[z, x]_1, y]_2+\rho_2([z, x]_1)v-\rho_2(y)V_1 \Big) \\
			&+k_2k_1 \Big([[z, x]_2, y]_1+\rho_1([z, x]_2)v-\rho_1(y)V_2 \Big) 
			+k_2^2 \Big([[z, x]_2, y]_2+\rho_2([z, x]_2)v-\rho_2(y)V_2 \Big)	
		\end{align*}
		where $V_1=\rho_1(z)u-\rho_1(x)w$ and $V_2=\rho_2(z)u-\rho_2(x)w$.\\
		In the first step, we have 
		\begin{align*}
			[[x, y]_1, z]_1+c.p.=0 \text{ and } [[x, y]_2, z]_2+c.p.=0
		\end{align*}
		as $(L, [\cdot, \cdot]_1, [\cdot, \cdot, \cdot]_1)$ and $(L, [\cdot, \cdot]_2, [\cdot, \cdot, \cdot]_2)$ are two Lie Yamaguti algebras, that lead us to the following sum 
		\begin{align*}
			&\{\{x+u, y+v\}_C, z+w\}_C+\{\{y+v, z+w\}_C, x+u\}_C+\{\{z+w, x+u\}_C, y+v\}_C \\
			=&k_1^2 \Big(\rho_1([x, y]_1)w-\rho_1(z)W_1 \Big)+k_1k_2 \Big([[x, y]_1, z]_2+\rho_2([x, y]_1)w-\rho_2(z)W_1 \Big) \\
			&+k_2k_1 \Big([[x, y]_2, z]_1+\rho_1([x, y]_2)w-\rho_1(z)W_2 \Big) + k_2^2 \Big(\rho_2([x, y]_2)w-\rho_2(z)W_2 \Big) \\
			&+k_1^2 \Big(\rho_1([y, z]_1)u-\rho_1(x)U_1 \Big) 
			+k_1k_2 \Big([[y, z]_1, x]_2+\rho_2([y, z]_1)u-\rho_2(x)U_1 \Big) \\
			&+k_2k_1 \Big([[y, z]_2, x]_1+\rho_1([y, z]_2)u-\rho_1(x)U_2 \Big) 
			+k_2^2 \Big(\rho_2([y, z]_2)u-\rho_2(x)U_2 \Big) \\
			&+k_1^2 \Big(\rho_1([z, x]_1)v-\rho_1(y)V_1 \Big) 
			+k_1k_2 \Big([[z, x]_1, y]_2+\rho_2([z, x]_1)v-\rho_2(y)V_1 \Big) \\
			&+k_2k_1 \Big([[z, x]_2, y]_1+\rho_1([z, x]_2)v-\rho_1(y)V_2 \Big) 
			+k_2^2 \Big(\rho_2([z, x]_2)v-\rho_2(y)V_2 \Big). 
		\end{align*}
		On the other hand,  by $(CLY1)$, we have 
		\begin{align*}
			[[x, y]_1, z]_2+c.p.+[[x, y]_2, z]_1+c.p.=0.
		\end{align*}
		This leads us to the following sum 
		\begin{align*}
			&\{\{x+u, y+v\}_C, z+w\}_C+\{\{y+v, z+w\}_C, x+u\}_C+\{\{z+w, x+u\}_C, y+v\}_C \\
			=&k_1^2 \Big(\rho_1([x, y]_1)w-\rho_1(z)W_1 \Big)+k_1k_2 \Big(\rho_2([x, y]_1)w-\rho_2(z)W_1 \Big) \\
			&+k_2k_1 \Big(\rho_1([x, y]_2)w-\rho_1(z)W_2 \Big) + k_2^2 \Big(\rho_2([x, y]_2)w-\rho_2(z)W_2 \Big) \\
			&+k_1^2 \Big(\rho_1([y, z]_1)u-\rho_1(x)U_1 \Big) 
			+k_1k_2 \Big(\rho_2([y, z]_1)u-\rho_2(x)U_1 \Big) \\
			&+k_2k_1 \Big(\rho_1([y, z]_2)u-\rho_1(x)U_2 \Big) 
			+k_2^2 \Big(\rho_2([y, z]_2)u-\rho_2(x)U_2 \Big) \\
			&+k_1^2 \Big(\rho_1([z, x]_1)v-\rho_1(y)V_1 \Big) 
			+k_1k_2 \Big(\rho_2([z, x]_1)v-\rho_2(y)V_1 \Big) \\
			&+k_2k_1 \Big(\rho_1([z, x]_2)v-\rho_1(y)V_2 \Big) 
			+k_2^2 \Big(\rho_2([z, x]_2)v-\rho_2(y)V_2 \Big) .
		\end{align*}
		In the last step,  by Eq. \eqref{eq1.3}, we obtain
		\begin{align*}
			&\{\{x+u, y+v\}_C, z+w\}_C+\{\{y+v, z+w\}_C, x+u\}_C+\{\{z+w, x+u\}_C, y+v\}_C \\
			=&k_1^2 \Big(\rho_1(x)\rho_1(y)w-\rho_1(y)\rho_1(x)w-\rho_1(z)\rho_1(x)v+\rho_1(z)\rho_1(y)u\Big) \\
			&+k_2^2 \Big(\rho_2(x)\rho_2(y)w-\rho_2(y)\rho_2(x)w-\rho_2(z)\rho_2(x)v+\rho_2(z)\rho_2(y)u\Big) \\
			&+k_1k_2 \Big( \rho_1(x)\rho_2(y)w-\rho_1(y)\rho_2(x)w+\rho_2(x)\rho_1(y)w-\rho_2(y)\rho_1(x)w \\
			&-\rho_2(z)\rho_1(x)v+\rho_2(z)\rho_1(y)u-\rho_1(z)\rho_2(x)v+\rho_1(z)\rho_2(y)u
			\Big) \\
			&+k_1^2 \Big(\rho_1(z)\rho_1(x)v-\rho_1(x)\rho_1(z)v-\rho_1(y)\rho_1(z)u+\rho_1(y)\rho_1(x)w\Big) \\
			&+k_2^2 \Big(\rho_2(z)\rho_2(x)v-\rho_2(x)\rho_2(z)v-\rho_2(y)\rho_2(z)u+\rho_2(y)\rho_2(x)w\Big) \\
			&+k_1k_2 \Big( \rho_1(z)\rho_2(x)v-\rho_1(x)\rho_2(z)v+\rho_2(z)\rho_1(x)v-\rho_2(x)\rho_1(z)v \\
			&-\rho_2(y)\rho_1(z)u+\rho_2(y)\rho_1(x)w-\rho_1(y)\rho_2(z)u+\rho_1(y)\rho_2(x)w
			\Big) \\
			&+k_1^2 \Big(\rho_1(y)\rho_1(z)u-\rho_1(z)\rho_1(y)u-\rho_1(x)\rho_1(y)w+\rho_1(x)\rho_1(z)v\Big) \\
			&+k_2^2 \Big(\rho_2(y)\rho_2(z)u-\rho_2(z)\rho_2(y)u-\rho_2(x)\rho_2(y)w+\rho_2(x)\rho_2(z)v\Big) \\
			&+k_1k_2 \Big( \rho_1(y)\rho_2(z)u-\rho_1(z)\rho_2(y)u+\rho_2(y)\rho_1(z)u-\rho_2(z)\rho_1(y)u \\
			&-\rho_2(x)\rho_1(y)w+\rho_2(x)\rho_1(z)v-\rho_1(x)\rho_2(y)w+\rho_1(x)\rho_2(z)v
			\Big) \\
			=&0.
		\end{align*}
		It implies that
		\begin{equation}
			\{\{x+u, y+v\}_C, z+w\}_C+\{\{y+v, z+w\}_C, x+u\}_C+\{\{z+w, x+u\}_C, y+v\}_C=0.
			\label{eq1.4}
		\end{equation}
		On the other hand, we have 
		\begin{align*}
			&\{x+u, y+v, z+w\}_C+\{y+v, z+w, x+u\}_C+\{z+w, x+u, y+v\}_C\\
			=&k_1\Big([x, y, z]_1+D_1(x, y)w+\mu_1(y, z)u-\mu_1(x, z)v 
			\Big) 
			+k_2 \Big([x, y, z]_2+D_2(x, y)w+\mu_2(y, z)u-\mu_2(x, z)v 
			\Big) \\ 
			&+k_1\Big([y, z, x]_1+D_1(y, z)u+\mu_1(z, x)v-\mu_1(y, x)w 
			\Big) 
			+k_2\Big([y, z, x]_2+D_2(y, z)u+\mu_2(z, x)v-\mu_2(y, x)w 
			\Big) \\ 
			&+k_1\Big([z, x, y]_1+D_1(z, x)v+\mu_1(x, y)w-\mu_1(z, y)w 
			\Big)+k_2\Big([z, x, y]_2+D_2(z, x)v+\mu_2(x, y)w-\mu_2(z, y)w 
			\Big).
		\end{align*}
		As $(L, [\cdot, \cdot]_1, [\cdot, \cdot, \cdot]_1)$ and $(L, [\cdot, \cdot]_2, [\cdot, \cdot, \cdot]_2)$ are two Lie Yamaguti algebras, we have 
		\begin{align*}
			&[x, y, z]_1+c.p.=0 \\
			&[x, y, z]_2+c.p.=0
		\end{align*}
		This leads us to reduce the previous sum as follows 
		\begin{align*}
			&\{x+u, y+v, z+w\}_C+\{y+v, z+w, x+u\}_C+\{z+w, x+u, y+v\}_C\\
			=&k_1\Big(D_1(x, y)w+\mu_1(y, z)u-\mu_1(x, z)v 
			\Big) 
			+k_2 \Big(D_2(x, y)w+\mu_2(y, z)u-\mu_2(x, z)v 
			\Big) \\ 
			&+k_1\Big(D_1(y, z)u+\mu_1(z, x)v-\mu_1(y, x)w 
			\Big) 
			+k_2\Big(D_2(y, z)u+\mu_2(z, x)v-\mu_2(y, x)w 
			\Big) \\ 
			&+k_1\Big(D_1(z, x)v+\mu_1(x, y)w-\mu_1(z, y)w 
			\Big)+k_2\Big(D_2(z, x)v+\mu_2(x, y)w-\mu_2(z, y)w 
			\Big).	
		\end{align*}
		Moreover, we have  $D_1(x, y)=\mu_1(y, x)-\mu_1(x, y)$ and $D_2(x, y)=\mu_2(y, x)-\mu_2(x, y)$. After some computations, we obtain 
		\begin{equation}
			\label{eq1.5}
			\{x+u, y+v, z+w\}_C+\{y+v, z+w, x+u\}_C+\{z+w, x+u, y+v\}_C=0.	
		\end{equation}
		Hence, $(LY1)$ is verified by using Eqs. \eqref{eq1.4} and  \eqref{eq1.5}. Where $(LY2)$, $(LY3)$ and $(LY4)$ can be prove similarly. 
	\end{proof}
	\begin{re}
		If $(L\oplus V, \{\cdot, \cdot\}_C, \{\cdot, \cdot, \cdot\}_C)$ reduces to a compatible Lie algebras $(L\oplus V, \{\cdot, \cdot\}_C)$ with $\{\cdot, \cdot\}_C=k_1\{\cdot, \cdot\}_1+k_2\{\cdot, \cdot\}_2$. Then we obtain a semidirect compatible Lie algebra $(L, [\cdot, \cdot]_1, [\cdot, \cdot]_2)$ with the a representation $(V;\rho)$, where $\rho=k_1 \rho_1+k_2 \rho_2$. \\
		Also If $(L\oplus V, \{\cdot, \cdot\}_C, \{\cdot, \cdot, \cdot\}_C)$ reduces to a compatible Lie triple system $(L\oplus V, \{\cdot, \cdot, \cdot\}_C)$ with $\{\cdot, \cdot, \cdot\}_C=k_1(\{\cdot, \cdot, \cdot\}_1)+k_2(\{\cdot, \cdot, \cdot\}_2)$. Then we obtain a semidirect compatible Lie triple system $(L, [\cdot, \cdot, \cdot]_1, [\cdot, \cdot, \cdot]_2)$ with the a representation $(V;\mu)$ where $\mu=k_1 \mu_1+k_2 \mu_2$.
	\end{re}
	Derivation of a compatible Lie Yamaguti algebras is given as follows
	\begin{defn}
		A derivation of a compatible Lie Yamaguti algebras $(L, \lceil \cdot, \cdot \rceil, \lceil \cdot, \cdot, \cdot \rceil)$ is a linear map $\delta : L \rightarrow L$, such that $\delta$ is a derivation for both $(L, [\cdot, \cdot]_1, [\cdot, \cdot, \cdot]_1)$ and $(L, [\cdot, \cdot]_2, [\cdot, \cdot, \cdot]_2)$ and verify the following conditions 
		\begin{align*}
			\delta (\lceil \cdot, \cdot \rceil)&=k_1 \delta ([\cdot, \cdot]_1)+k_2 \delta ([\cdot, \cdot]_2) \\
			\delta (\lceil \cdot, \cdot, \cdot \rceil)&=k_1 \delta([\cdot, \cdot, \cdot]_1)+k_2 \delta ([\cdot, \cdot, \cdot]_2)
		\end{align*}
		where 
		\begin{align*}
			\lceil \cdot, \cdot \rceil&=k_1 ([\cdot, \cdot]_1)+k_2 ([\cdot, \cdot]_2) \\
			\lceil \cdot, \cdot, \cdot \rceil&=k_1([\cdot, \cdot, \cdot]_1)+k_2 ([\cdot, \cdot, \cdot]_2)
		\end{align*}
	\end{defn} 
	\begin{ex}
		For a compatible Lie Yamaguti algebras $(L, \lceil \cdot, \cdot \rceil, \lceil \cdot, \cdot, \cdot \rceil)$, the linear map defined by
		\begin{align*}
			T_{(a, b)}=k_1 T_{1_{(a, b)}}+k_2 T_{2_{(a, b)}}:&L \rightarrow L \\
			&z \mapsto \lceil a, b, z \rceil = k_1[a, b, z]_1+k_2 [a, b, z]_2~~for ~~~ a, b, z_1, z_2, z_3 \in L 
		\end{align*} 
		is a derivation of $(L, \lceil \cdot, \cdot \rceil, \lceil \cdot, \cdot, \cdot \rceil).$ \\
		\newline 
		For simplification, we write $T_1$ at the place of $T_{1_{(a, b)}}$ and $T_2$ at the place of $T_{2_{(a, b)}}$, 
		\begin{align*}
			T_{(a, b)}(\lceil x, y \rceil)&=(k_1 T_1+k_2 T_2)\Big( k_1 ([x, y]_1)+k_2 ([x, y]_2) \Big) \\
			&=k_1^2 T_1[x, y]_1+k_2^2[x, y]_2+k_1k_2 T_1[x, y]_2+k_2k_1 T_2[x, y]_1\\
			& \text{ by (CLY5) and proposition \eqref{prop3.1}, we have } \\
			T_{(a, b)}(\lceil x, y \rceil)&=k_1^2[T_1(x), y]_1+k_1^2[x, T_1(y)]_2+k_2^2[T_2(x), y]_2+k_2^2[x, T_2(y)]_2 \\
			&+k_1k_2 \Big([T_1(x), y]_2+ [x, T_1(y)]_2+[T_2(x), y]_1+[x, T_2(y)]_1 \Big) \\
			&=\lceil T_{(a, b)}(x), y \rceil + \lceil x, T_{(a, b)}(y) \rceil.
		\end{align*}
		Similarly,  using (CLY4) and proposition \eqref{prop3.1}, we have
		\begin{equation*}
			T_{(a, b)}(\lceil x, y, z \rceil)=\lceil T_{(a, b)}(x), y, z \rceil +\lceil x, T_{(a, b)}(y), z \rceil +\lceil x, y, T_{(a, b)}(z) \rceil
		\end{equation*}
		which provides $T_{(a, b)}$ as a derivation of $(L, \lceil \cdot, \cdot \rceil, \lceil \cdot, \cdot, \cdot \rceil)$ 
	\end{ex}
\section{Maurer-Cartan characterization and cohomology of compatible Lie Yamaguti algebras} \label{section cohomolgy}
In this section, we introduce the Maurer-Cartan elements of compatible Lie Yamaguti algebras according to the bidifferential graded Lie algebras. We also introduce the cohomology theory of compatible Lie Yamaguti algebras.
\subsection*{Maurer-Cartan characterization of compatible Lie Yamaguti algebras}
In this subsection, we introduce the Maurer-Cartan element of compatible Lie Yamaguti algebras.\\
Let us first recall some notions and basic results from \cite{Z4, L1}. \\
A degree $1$ element $x\in L_1$ is called a Maurer-Cartan element of a differential graded Lie algebra $(L=\oplus_{\mathrm{k}\in\mathbb{Z}}L_k,[\cdot,\cdot],\partial) $ if it satisfies the Maurer-Cartan equation 
\begin{equation*}
	\partial x+\frac{1}{2}[x,x]=0,\quad \forall x\in L.
\end{equation*}
Note that a graded Lie algebra is a special case of differential graded Lie algebra with $\partial=0$. Then, for a graded Lie algebra $(L=\oplus_{\mathrm{k}\in\mathbb{Z}}L_k,[\cdot,\cdot],\partial) $, an element $x\in L_1$ such that $[x,x]=0$ is a Marer-Cartan element of $L$.\\
A permutation $\sigma\in S_n$ is called an $(i, n-i)$-shuffle if $\sigma(1)<\ldots<\sigma(i)$ and $\sigma(i+1)<\ldots<\sigma(n)$. If $i=0$ or $i=n$, we assume $\sigma=\mathrm{id}$, the set of $(i,n-i)$ shuffles is denoted by $\mathcal{S}_{(i,n-i)}$.\\
The Maurer-Cartan element of Lie Yamaguti algebras was introduced in \cite{Z4}. Let $L$ be a vector space,  we have $$\mathfrak{C}^p(L,L):= \left\{
\begin{array}{ll}
	\mathrm{Hom}((\wedge^2L)\otimes\cdots\otimes(\wedge^2L),L)\oplus \mathrm{Hom}((\wedge^2L)\otimes\cdots\otimes(\wedge^2L)\otimes L,L) & \text{ for } p\geq1,\\
	\mathrm{Hom}(L,L)& \text{ for } p=0.	
\end{array}
\right.
$$ 
Then $\mathfrak{C}^\star(L, L)=\oplus_{p\geq0}\mathfrak{C}^p(L, L)$ is a graded vector space where the degree of elements in $\mathfrak{C}^p(L, L)$ is $p$.\\
For $P=(P_1,P_2)\in \mathfrak{C}^p(L,L)$ and $Q=(Q_1,Q_2)\in \mathfrak{C}^q(L,L)$ $(p,q\geq1)$, we define
$P\circ Q:=((P\circ Q)_1,(P\circ Q)_2)$.
\begin{align*}
&	(P\circ Q)_1(X_1,\cdots,X_{p+q})\\
&=\displaystyle\sum_{\sigma\in\mathcal{S}_{(p,q)},\sigma(p+q)=p+q}(-1)^{pq}\mathrm{sign}(\sigma)P_2(X_{\sigma(1)},\cdots,X_{\sigma(p)},Q_1(X_{\sigma(p+1)},\cdots,X_{\sigma(p+q)}))\\
	&+\displaystyle\sum_{k=1}^p(-1)^{(k-1)q}\displaystyle\sum_{\sigma\in\mathcal{S}_{(k-1,q)}}\mathrm{sign}(\sigma)P_1(X_{\sigma(1)},\cdots,X_{\sigma(k-1)},x_{q+k}\wedge Q_2(X_{\sigma(k)},\cdots,X_{\sigma(k+q-1)},y_{k+q})
	,X_{k+q+1},\cdots,X_{p+q})\\
	&+\displaystyle\sum_{k=1}^p(-1)^{(k-1)q}\displaystyle\sum_{\sigma\in\mathcal{S}_{(k-1,q)}}\mathrm{sign}(\sigma)P_1(X_{\sigma(1)},\cdots,X_{\sigma(k-1)}, Q_2(X_{\sigma(k)},\cdots,X_{\sigma(k+q-1)},x_{k+q})\wedge y_{k+q}
	,X_{k+q+1},\cdots,X_{p+q})
\end{align*}
and 
\begin{align*}
	&(P\circ Q)_2(X_1,\cdots,X_{p+q},x)=\\
	&\displaystyle\sum_{\sigma\in\mathcal{S}_{(p,q)}}(-1)^{pq}\mathrm{sign(\sigma)}P_2(X_{\sigma(1)},\cdots,X_{\sigma(p)},Q_2(X_{\sigma(p+1)},\cdots,X_{\sigma(p+q)},x))\\
	&+\displaystyle\sum_{k=1}^p(-1)^{(k-1)q}\displaystyle\sum_{\sigma\in\mathcal{S}_{(k-1,q)}}\mathrm{sign(\sigma)}P_2(X_{\sigma(1)},\cdots,X_{\sigma(k-1)},x_{q+k}\wedge Q_2(X_{\sigma(k)},\cdots,X_{\sigma(k+q-1)},y_{k+q}),X_{k+q+1},\cdots,X_{p+q},x)\\
	&+\displaystyle\sum_{k=1}^p(-1)^{(k-1)q}\displaystyle\sum_{\sigma\in\mathcal{S}_{(k-1,q)}}\mathrm{sign(\sigma)}P_2(X_{\sigma(1)},\cdots,X_{\sigma(k-1)}, Q_2(X_{\sigma(k)},\cdots,X_{\sigma(k+q-1)},x_{k+q})\wedge y_{k+q},X_{k+q+1},\cdots,X_{p+q},x).
\end{align*}
In the sequel we denote $\pi=[\cdot,\cdot]$,  $\omega=[\cdot,\cdot,\cdot]$ and $\Pi=(\pi,\omega)$.
\begin{thm} \cite{Z4} \label{thm MC of LYA}
	With the above notations, the graded vector space $\mathfrak{C}^\star(L,L)$ equipped with the graded commutator
	\begin{equation*}
		[P,Q]_\mathrm{LieY}=P\circ Q-(-1)^{pq} Q\circ P
	\end{equation*}
is a graded Lie algebra. Moreover, if $\Pi\in\mathfrak{C}^1(L,L)$ defines a Lie Yamaguti algebra structure on $L$ then $\Pi$ is a Maurer-Cartan element of the graded Lie algebra $(\mathfrak{C}^\star(L,L),[\cdot,\cdot]_\mathrm{LieY})$
\end{thm}
\begin{re}
	Let $\Pi=(\pi,\omega)\in\mathfrak{C}^1(L,L)$, then $(L,\pi,\omega)$ does not mean a Lie Yamaguti algebra. According to (\cite{Z4}, Theorem 3.1) $\Pi$ is a Maurer-Cartan element if and only if $\Pi$ satisfies the following equations 
	\begin{eqnarray*}
		\omega(x,y,\pi(z,w))&=&\pi(\omega(x,y,z),w)+\pi(z,\omega(x,y,w)),\\
		\omega(x,y,\omega(z,w,t))&=&\omega(\omega(x,y,z),w,t)+\omega(z,\omega(x,y,w),t)+\omega(z,w,\omega(x,y,t)).
	\end{eqnarray*}
\end{re}
Let $\Pi$ define a Lie Yamaguti algebra structure on the vector space $L$. It follows from the graded Jacobi identity that $\partial_\Pi:=[\Pi,\cdot]_\mathrm{LieY}$ is a differential on $(\mathfrak{C}^\star(L,L),[\cdot,\cdot]_\mathrm{LieY})$.
\begin{thm}\cite{Z4}
	Let $(L,\pi,\omega)$ be a Lie Yamaguti algebra. Then, the triple $(\mathfrak{C}^\star(L,L),[\cdot,\cdot]_\mathrm{LieY},\partial_\Pi)$ is a differential graded Lie algebra with $\partial_\Pi:=[\Pi,\cdot]_\mathrm{LieY}$.
\end{thm}
The notion of bidifferential graded Lie algebra and its Maurer-Cartan elements was introduced in \cite{L1}. Let us recall some facts that we need in the construction of Maurer-Cartan elements of compatible Lie Yamaguti algebra.
\begin{defn}\cite{L1}
	Let $(L,[\cdot,\cdot],\partial_1)$ and $(L,[\cdot,\cdot],\partial_2)$ be two differential graded Lie algebras. Then the quadruple $(L,[\cdot,\cdot],\partial_1,\partial_2)$ is called a bidifferential graded Lie algebra if $\partial_1$ and $\partial_2$ satisfies 
	\begin{equation*}
		\partial_1\circ\partial_2+\partial_2\circ\partial_1=0.
	\end{equation*}
\end{defn}
\begin{defn}\cite{L1}
	Let $(L,[\cdot,\cdot],\partial_1,\partial_2)$ be a bidifferential graded Lie algebra. A couple $(\pi_1,\pi_2)\in L_1\oplus L_1$ is called a Maurer-Cartan element of the bidifferntial graded Lie algebra $(L,[\cdot,\cdot],\partial_1,\partial_2)$ if $\pi_1$ and $\pi_2$ are Maurer-Cartan elements of differential graded Lie algebras $(L,[\cdot,\cdot],\partial_1)$ and $(L,[\cdot,\cdot],\partial_2)$ respectively and 
	\begin{equation*}
		\partial_1\pi_2+\partial_2\pi_1+[\pi_1,\pi_2]=0.
	\end{equation*}
\end{defn}
Now we are positioned to introduce the Maurer-Cartan characterization of compatible Lie Yamaguti algebras.
\begin{thm}\label{MC of compatible LYA}
	Let $L$ be a vector space and $\Pi_1,\Pi_2\in\mathfrak{C}^1(L,L)$. If $(L,\pi_1,\omega_1,\pi_2,\omega_2)$ is a compatible Lie Yamaguti algebra then $(\Pi_1,\Pi_2)$ is a Maurer-Cartan element of the bidifferential graded Lie algebra $(\mathfrak{C}^\star(L,L),[\cdot,\cdot]_\mathrm{LieY};\partial_1=0,\partial_2=0)$
\end{thm}
\begin{proof}
	By theorem \eqref{thm MC of LYA} $\Pi_1,\Pi_2 \in \mathrm{C}^1(L,L)$ define Lie Yamaguti algebra structure on $L$ if 
	\begin{equation*}
		[\Pi_1,\Pi_1]_\mathrm{LieY}=0,\quad [\Pi_2,\Pi_2]_\mathrm{LieY}=0.
	\end{equation*}
Moreover, $[\Pi_1,\Pi_2]_\mathrm{LieY}=([\Pi_1,\Pi_2]_1,[\Pi_1,\Pi_2]_2)$, where \\
$[\Pi_1,\Pi_2]_1=(\Pi_1\circ\Pi_2)_1+(\Pi_2\circ\Pi_1)_1$ which means by direct computation
\begin{align*}
	[\Pi_1,\Pi_2]_1(x,y,z,w)&=(\Pi_1\circ\Pi_2)_1(x,y,z,w)+(\Pi_2\circ\Pi_1)_1(x,y,z,w)\\
	&=\omega_1(x,y,\pi_2(z,w))-\pi_1(z,\omega_2(x,y,w))-\pi_1(\omega_2(x,y,z),w)\\
	&-\omega_2(x,y,\pi_1(z,w))+\pi_2(z,\omega_1(x,y,w))+\pi_2(\omega_1(x,y,z),w)\\
	&\overset{\textbf{(CLY2)}}{=}0.
\end{align*}
And 
\begin{align*}
	[\Pi_1,\Pi_2]_2(x,y,z,w,t)&=(\Pi_1\circ\Pi_2)_2(x,y,z,w,t)+(\Pi_2\circ\Pi_1)_2(x,y,z,w,t)\\
	&=\omega_2(x,y,\omega_1(z,w,t))-\omega_2(\omega_1(x,y,z),w,t)-\omega_2(z,\omega_1(x,y,w),t)-\omega_2(z,w,\omega_1(x,y,t))\\
	&+\omega_1(x,y,\omega_2(z,w,t))-\omega_1(\omega_2(x,y,z),w,t)-\omega_1(z,\omega_2(x,y,w),t)-\omega_1(z,w,\omega_2(x,y,t))\\ &\overset{\textbf{(CLY3)}}{=}0.
\end{align*}
This completes the proof.
\end{proof}
\begin{re}\label{remark0}
Recall that from \cite{L1}, if $(\pi_1,\pi_2)$ is a Maurer-Cartan element of a bidifferential graded Lie algebra $(L,[\cdot,\cdot], \partial_1, \partial_2)$. Then $(L,[\cdot,\cdot],\widetilde{\partial_1},\widetilde{\partial_2})$ is a bidifferential graded Lie algebra with  \begin{equation*}\widetilde{\partial_1}P:=\partial_1P+[\pi_1,P],\quad \widetilde{\partial_2}P:=\partial_2P+[\pi_2,P],\quad \forall~ P\in L.\end{equation*}
As a conclusion, let $\widetilde{\pi_1},\widetilde{\pi_2}\in L_1$ and $(\pi_1+\widetilde{\pi_1},\pi_2+\widetilde{\pi_2})$ is a Maurer-Cartan element of the bidifferential graded Lie algebra $(L,[\cdot,\cdot],\partial_1,\partial_2)$ if and only if $(\widetilde{\pi_1},\widetilde{\pi_2})$ is a Maurer-Cartan element of the bidifferential graded Lie algebra $(L,[\cdot,\cdot],\widetilde{\partial_1},\widetilde{\partial_2})$.
\end{re}
Now we are ready to give the bidifferential graded Lie algebra that controls deformations of a compatible Lie Yamaguti algebra.
\begin{thm}
	Let $(L,\pi_1,\omega_1,\pi_2,\omega_2)$ be a compatible Lie Yamaguti algebra. Then the following two conclusions hold
	\begin{enumerate}
		\item[1)] $(\mathfrak{C}^\star(L,L),[\cdot,\cdot]_\mathrm{LieY},\widetilde{\partial_1},\widetilde{\partial_2})$ is a bidifferential graded Lie algebra with 
		\begin{equation*}
	\widetilde{\partial_1}=[\pi_1,P]_\mathrm{LieY},\quad \widetilde{\partial_2}=[\pi_2,P]_\mathrm{LieY}.
		\end{equation*}
	\item [2)] For any $\widetilde{\Pi_1},\widetilde{\Pi_1}\in\mathfrak{C}^1(L,L)$, $(L,\pi_1+\widetilde{\pi_1},\omega_1+\widetilde{\omega_1},\pi_2+\widetilde{\pi_2},\omega_2+\widetilde{\omega_2})$ be compatible Lie Yamaguti algebra then $(\widetilde{\Pi_1},\widetilde{\Pi_2})$ is a Maurer-Cartan element of $(\mathfrak{C}^\star(L,L),[\cdot,\cdot]_\mathrm{LieY},\widetilde{\partial_1},\widetilde{\partial_2})$.\end{enumerate}
\end{thm}
\begin{proof}
	The proof results from \eqref{thm MC of LYA}
and the previous remark \eqref{remark0}. For more details see (\cite{L1}, theorem 3.12). 
\end{proof}
\subsection*{Cohomology of compatible Lie Yamaguti algebras}
In this subsection we introduce the cohomology of compatible Lie Yamaguti algebra. \\
Recall that the cohomology theory of Lie Yamaguti algebra was introduced in \cite{K3}. Let $(L,[\cdot,\cdot],[\cdot,\cdot,\cdot])$ be a Lie Yamaguti algebra and $(V;\rho,\mu)$ a representation of it. Denote by $C_\mathrm{LieY}^p(L,V)$ for $p\geq1$ the set of $p$-cochains such that 
	$$
C_\mathrm{LieY}^p(L,V):= \left\{
\begin{array}{ll}
	\mathrm{Hom}((\wedge^2L)\otimes\cdots(\wedge^2L),V)\times \mathrm{Hom}((\wedge^2L)\otimes\cdots(\wedge^2L)\otimes L,V) & \text{ for } n\geq1,\\
	\mathrm{Hom}(L,V)& \text{ for } n=0.	
\end{array}
\right.
$$ 
For $p\geq1$, the coboundary operator $\delta:C_\mathrm{LieY}^p(L,V)\rightarrow C_\mathrm{LieY}^{p+1}(L,V)$ is defined as follows :
\begin{enumerate}
	\item [$\bullet$] If $n\geq1$, for any $F=(f,g)\in C_\mathrm{LieY}^{n+1}(L,V)$, the coboundary map 
	\begin{eqnarray*}
		\delta=(\delta_1,\delta_2):C_\mathrm{LieY}^{n+1}(L,V)&\rightarrow& C_\mathrm{LieY}^{n+2}(L,V)\\
		F&\mapsto&(\delta_1F,\delta_2F)
	\end{eqnarray*}
is given by 
\begin{align*}
	&(\delta_1F)(X_1,\cdots,X_{n+1})=\\
	&=(-1)^n\Big(\rho(x_{n+1})g(X_1,\cdots,X_n,y_{n+1})-\rho(y_{n+1})g(X_1,\cdots,X_n,x_{n+1})-g(X_1,\cdots,X_n,[x_{n+1},y_{n+1}]) \Big)\\
	&+\displaystyle\sum_{k=1}^n(-1)^{(k+1)}D(X_k)f(X_1,\cdots,\hat{X_k},\cdots,X_{n+1})+\displaystyle\sum_{1\leq k<l\leq n+1}(-1)^{k}f(X_1,\cdots,\hat{X_k},\cdots,X_k\circ X_l,\cdots,X_{n+1}),
\end{align*}
and 
\begin{align*}
	&(\delta_2F)(X_1,\cdots,X_{n+1},z)=\\
	&(-1)^n\Big(\mu(y_{n+1},z)g(X_1,\cdots,X_n,x_{n+1})-\mu(x_{n+1},z)g(X_1,\cdots,X_n,y_{n+1}) \Big)\\
	&+\displaystyle\sum_{k=1}^{n+1}(-1)^{(k+1)}D(X_k)g(X_1,\cdots,\hat{X_k},\cdots,X_{n+1},z)+\displaystyle\sum_{1\leq k<l\leq n+1}(-1)^{k}g(X_1,\cdots,\hat{X_k},\cdots,X_k\circ X_l,\cdots,X_{n+1},z)\\
	&+\displaystyle\sum_{k=1}^{n+1}(-1)^kg(X_1,\cdots,\hat{X_k},\cdots,X_{n+1},[x_k,y_k,z]).
\end{align*}
Where $X_i=x_i\wedge y_i\in\wedge^2L$, $(i=1,\cdots,n+1)$, $z\in L$ and $X_k\circ X_l:=[x_k,y_k,x_l]\wedge y_l+x_l\wedge[x_k,y_k,y_l]$.
\item [$\bullet$] For the case that $n=0$, any element $f\in C_\mathrm{LieY}^1(L,V)$ given the coboundary map 
\begin{eqnarray*}
	\delta:C_\mathrm{LieY}^1(L,V)&\rightarrow& C_\mathrm{LieY}^2(L,V)\\
	f&\mapsto&(\delta_1f,\delta_2f)
\end{eqnarray*}
is given by 
\begin{align*}
	(\delta_1f)(x,y)&=\rho(x)f(y)-\rho(y)f(x)-f([x,y]),\\
	(\delta_2f)(x,y,z)&=D(x,y)f(z)+\mu(y,z)f(x)-\mu(x,z)f(y)-f([x,y,z]),\quad \forall x,y,z\in L.
\end{align*}
\end{enumerate}
Yamaguti showed that $\delta$ is a differential, which means that $\delta\circ\delta=0$. For more details, any $f\in C_\mathrm{LieY}^1(L,V)$ we have 
\begin{equation*}
	\delta_1(\delta_1f,\delta_2f)=0 \text{ and } \delta_2(\delta_1f,\delta_2f)=0.
\end{equation*} 
Moreover, for all $F\in C_\mathrm{LieY}^p(L,V)$ $(p\geq2)$, we have 
\begin{equation*}
	\delta_1(\delta_1F,\delta_2F)=0 \text{ and } \delta_2(\delta_1F,\delta_2F)=0.
\end{equation*}
Then the cochain complex $\Big(C_\mathrm{LieY}^\star(L,V)=\oplus_{p\geq1}C_\mathrm{LieY}^p(L,V),\delta\Big)$ is well defined.\\
According to the Maurer-Cartan characterization of Lie Yamaguti algebras, we obtain the following theorem
\begin{thm}\cite{L1}
	Let $(L,\pi,\omega)$ be a Lie Yamaguti algebra and $\delta:C_\mathrm{LieY}^n(L,L)\rightarrow C_\mathrm{LieY}^{n+1}(L,L)$ be the coboundary map associated tp the adjoint representation. Then we have 
	\begin{eqnarray*}
		(\delta F)&=&(-1)^n\partial_\Pi(F)=(-1)^n[\Pi,F]_\mathrm{LieY},\quad \forall F\in \mathfrak{C}^n(L,L) (n\geq1),\\
		\delta(f)&=&\partial_\Pi(f)=[\Pi,f]_\mathrm{LieY},\quad  \forall f\in\mathfrak{C}^0(L,L)=\mathrm{Hom}(L,L).
	\end{eqnarray*}
\end{thm}
Next we introduce the cohomology theory of compatible Lie Yamaguti algebras. Let $(L,[\cdot,\cdot]_1,[\cdot,\cdot,\cdot]_1,[\cdot,\cdot]_2,[\cdot,\cdot,\cdot]_2)$ be a compatible Lie Yamaguti algebra with 
\begin{equation*}
	\pi_1(x,y)=[x,y]_1,\quad \omega_1(x,y,z)=[x,y,z]_1,\quad \pi_2(x,y)=[x,y]_2,\quad \omega_2(x,y,z)=[x,y,z]_2,\quad \forall x,y,z\in L. 
\end{equation*} 
By theorem \eqref{MC of compatible LYA} $(\Pi_1=(\pi_1,\omega_1),\Pi_2=(\pi_2,\omega_2))$ is a Maurer-Cartan element of the bidifferential graded Lie algebra
$(\mathfrak{C}^\star(L,L),[\cdot,\cdot]_\mathrm{LieY};\partial_1=0,\partial_2=0)$.
Define the space of $0$-cochains to be $\mathfrak{C}_c^0=\mathrm{Hom}(L,L)$ and the space of $n$-cochains to be 
\begin{equation*}
	\mathfrak{C}^n_c(L,L):=\underbrace{\mathfrak{C}^n(L,L)\oplus\cdots\oplus\mathfrak{C}^n(L,L)}_{n-copies}
\end{equation*}
Define
\begin{eqnarray*}
	\delta_c:\mathfrak{C}_c^0(L,L)&\rightarrow &\mathfrak{C}_c^1(L,L)\\
	f&\mapsto& ([\Pi_1,f]_\mathrm{LieY},[\Pi_2,f]_\mathrm{LieY})=(\partial_{\Pi_1}(f),\partial_{\Pi_2}(f)),
\end{eqnarray*} 
and 
\begin{eqnarray*}
	\delta_c:\mathfrak{C}_c^n(L,L)&\rightarrow &\delta_c:\mathfrak{C}_c^{n+1}(L,L)\\
	(F_1,\cdots,F_n)&\mapsto& ([\Pi_1,F_1]_\mathrm{LieY},\cdots,\underbrace{[\Pi_2,F_{i-1}]_\mathrm{LieY}+[\Pi_1,F_i]_\mathrm{LieY}}_{\text{i-th place}},\cdots,[\Pi_2,F_n]_\mathrm{LieY}),
\end{eqnarray*} 
we have also 
\begin{equation*}
	\delta_c(F_1,\cdots,F_n):=(\partial_{\Pi_1}(F_1),\cdots,\partial_{\Pi_2}(F_{i-1})+\partial_{\Pi_1}(F_i),\cdots,\partial_{\Pi_2}(F_n))
\end{equation*}
where $(F_1,\cdots,F_n)\in\mathfrak{C}^n(L,L),\quad 2\leq i\leq n$.
\begin{thm}
	With the above notations, we have $\delta_c\circ\delta_c=0$, which means that $\Big(\mathfrak{C}_c^\star(L,L)=\oplus_{n\geq0}\mathfrak{C}_c^n(L,L),\delta_c^\star \Big)$ is a cochain complex.
\end{thm}
\begin{proof}
	For any map $f\in\mathfrak{C}_C^0(L,L)$, we have 
	\begin{equation*}
		\delta_c(\delta_c(f))=\delta_c(\partial_{\Pi_1}(f),\partial_{\Pi_2}(f))=(\partial_{\Pi_1}(\partial_{\Pi_1}(f)),\underbrace{\partial_{\Pi_1}(\partial_{\Pi_2}(f))+\partial_{\Pi_2}(\partial_{\Pi_1}(f))}_{0},\partial_{\Pi_2}(\partial_{\Pi_2}(f)))=0
	\end{equation*}
Moreover, for any $(F_1,\cdots,F_n)\in \mathfrak{C}_c^n(L,L)$, $n\geq1$ we get 
\begin{align*}
	\delta_c(\delta_c(F_1,\cdots,F_n))&=\delta_c\Big(\partial_{\Pi_1}(F_1),\cdots,\partial_{\Pi_2}(F_{i-1})+\partial_{\Pi_1}(F_i),\cdots,\partial_{\Pi_2}(F_n) \Big) \\
	&=\Big(\partial_{\Pi_1}(\partial_{\Pi_1}(F_1)),\partial_{\Pi_1}(\partial_{\Pi_1}(F_1))+\partial_{\Pi_2}(\partial_{\Pi_1}(F_1))+\partial_{\Pi_1}(\partial_{\Pi_1}(F_2)),\cdots,\\
	&\underbrace{\partial_{\Pi_2}(\partial_{\Pi_2}(F_{i-2}))+\partial_{\Pi_2}(\partial_{\Pi_1}(F_{i-1}))+\partial_{\Pi_1}(\partial_{\Pi_2}(F_{i-1}))+\partial_{\Pi_1}(\partial_{\Pi_1}(F_i))}_{3\leq i\leq n-1},\cdots,\\
	&\partial_{\Pi_2}(\partial_{\Pi_2}(F_{n-1}))+\partial_{\Pi_2}(\partial_{\Pi_1}(F_n))+\partial_{\Pi_1}(\partial_{\Pi_2}(F_n)),\partial_{\Pi_2}(\partial_{\Pi_2}(F_n)) \Big)=0.
\end{align*}
Thus we have $\delta_c\circ\delta_c=0$.
\end{proof}
\begin{defn}
	Let $(L,[\cdot,\cdot]_1,[\cdot,\cdot,\cdot]_1,[\cdot,\cdot]_2,[\cdot,\cdot,\cdot]_2)$ be a compatible Lie Yamaguti algebra. The cohomology of the cochain complex $(\mathfrak{C}_c^\star(L,L),\delta_c^\star)$ is called the cohomology of $(L,[\cdot,\cdot]_1,[\cdot,\cdot,\cdot]_1,[\cdot,\cdot]_2,[\cdot,\cdot,\cdot]_2)$ we denote the $n$-th cohomology group by $\mathcal{H}^n(L,L)$.
\end{defn}
\section{Infinitesimal deformations of compatible Lie Yamaguti algebras} \label{section deformation}
In this section, we study infinitesimal deformation of compatible Lie Yamaguti algebras.
\begin{defn}
	Let $(L,[\cdot,\cdot]_1,[\cdot,\cdot,\cdot]_1,[\cdot,\cdot]_2,[\cdot,\cdot,\cdot]_2)$ be a compatible Lie Yamaguti algebra and $\mu_i:L\times L\rightarrow L$ and $\lambda_i:L\times L\times L\rightarrow L$. Define 
	\begin{align*}
		[x,y]_t^i=&[x,y]_i+t\mu_i(x,y),\\
		[x,y,z]_t^i=&[x,y,z]_i+t\lambda_i(x,y,z),\quad \forall x,y,z\in L,\quad \forall i\in\{1,2\}.
	\end{align*}
If for any $t$, $(L,[\cdot,\cdot]_t^1,[\cdot,\cdot,\cdot]_t^1,[\cdot,\cdot]_t^2,[\cdot,\cdot,\cdot]_t^2)$ is a compatible Lie Yamaguti algebra then we say $(\mu_1,\lambda_1,\mu_2,\lambda_2)$ generates an infinitesimal deformation of $(L,[\cdot,\cdot]_1,[\cdot,\cdot,\cdot]_1,[\cdot,\cdot]_2,[\cdot,\cdot,\cdot]_2)$.
\end{defn}
We give some useful notations:
\begin{equation*}
	\Pi_i=(\pi_i,\omega_i),\quad \Pi_t^i=(\pi_t^i,\omega_t^i),\quad \pi_i(x,y)=[x,y]_i,\quad \omega_i(x,y,z)=[x,y,z]_i \text{ and } \pi_t^i=\pi_i+t\mu_i,\quad \omega_t^i=\omega_i+t\lambda_i.
\end{equation*}
By theorem \eqref{MC of compatible LYA} $(L,[\cdot,\cdot]_t^1,[\cdot,\cdot,\cdot]_t^1,[\cdot,\cdot]_t^2,[\cdot,\cdot,\cdot]_t^2)$ is an infinitesimal deformation of $(L,\pi_1,\omega_1,\pi_2,\omega_2)$ the we have the following 
\begin{equation}
	[\Pi_t^1,\Pi_t^1]_\mathrm{LieY}=0,\label{eqt compatible MC1}
\end{equation}
\begin{equation}
	[\Pi_t^2,\Pi_t^2]_\mathrm{LieY}=0,\label{eqt compatible MC2}
\end{equation}
\begin{equation}
	[\Pi_t^1,\Pi_t^2]_\mathrm{LieY}=0.\label{eqt compatible MC13}
\end{equation}
When we explore the equation \eqref{eqt compatible MC1} we obtain the following two equations 
\begin{equation}\label{equation1}
	\begin{split}
 &\omega_1(x,y,\mu_1(z,w))+\lambda_1(x,y,\pi_1(x,y,w))-\pi_1(z,\lambda_1(x,y,w))\\
 &-\mu_1(z,\omega_1(x,y,w))-\pi_1(\lambda_1(x,y,z),w)-\mu_1(\omega_1(x,y,z),w)=0,
\end{split}
\end{equation}
\begin{equation}\label{equation2}
	\begin{split}
		&\omega_1(x,y,\lambda_1(z,w,s))+\lambda_1(x,y,\omega_1(z,w,s))-\omega_1(y,\lambda_1(x,y,w),s)\\
		&-\lambda_1(y,\omega_1(x,y,w),s)-\omega_1(\lambda_1(x,y,z),w,s)-\lambda_1(\omega_1(x,y,z),w,s)\\
		&-\omega_1(z,w,\lambda_1(x,y,s))-\lambda_1(z,w,\omega(x,y,s))=0.
	\end{split}
\end{equation}
One can see that equations \eqref{equation1} and \eqref{equation2} are equivalent to 
\begin{equation}\label{equation3}
	[\Pi_1,(\mu_1,\lambda_1)]_\mathrm{LieY}=0.
\end{equation}
And similarly equation \eqref{eqt compatible MC2}, can be explored to the following 
\begin{equation}\label{equation4}
	\begin{split}
		&\omega_2(x,y,\mu_2(z,w))+\lambda_2(x,y,\pi_2(x,y,w))-\pi_2(z,\lambda_2(x,y,w))\\
		&-\mu_2(z,\omega_2(x,y,w))-\pi_2(\lambda_2(x,y,z),w)-\mu_2(\omega_2(x,y,z),w)=0,
	\end{split}
\end{equation}
\begin{equation}\label{equation5}
	\begin{split}
		&\omega_2(x,y,\lambda_2(z,w,s))+\lambda_2(x,y,\omega_2(z,w,s))-\omega_2(y,\lambda_2(x,y,w),s)\\
		&-\lambda_2(y,\omega_2(x,y,w),s)-\omega_2(\lambda_2(x,y,z),w,s)-\lambda_2(\omega_2(x,y,z),w,s)\\
		&-\omega_2(z,w,\lambda_2(x,y,s))-\lambda_2(z,w,\omega(x,y,s))=0.
	\end{split}
\end{equation}
So we obtain the following equation
\begin{equation}\label{equation6}
	[\Pi_2,(\mu_2,\lambda_2)]_\mathrm{LieY}=0.
\end{equation}
When we explore the last equation \eqref{eqt compatible MC13} $([\Pi_t^1,\Pi_t^2]_\mathrm{LieY}=0)$ we obtain the following two equations
\begin{equation}\label{equation7}
	\begin{split}
		&\omega_1(x,y,\mu_2(z,w))+\lambda_1(x,y,\pi_2(z,w))-\pi_1(z,\lambda_2(x,y,w))-\mu_1(z,\omega_2(x,y,w))\\
		&-\pi_1(\lambda_2(x,y,z,),w)-\mu_1(\omega_2(x,y,z),w)+\omega_2(x,y,\mu_1(z,w))+\lambda_2(x,y,\pi_1(z,w))\\
		&-\pi_2(z,\lambda_1(x,y,w))-\mu_2(z,\omega_1(x,y,w))-\pi_2(\lambda_1(x,y,z,),w)-\mu_2(\omega_1(x,y,z),w)=0,
	\end{split}
\end{equation} 
\begin{equation}\label{equation8}
	\begin{split}
		&\omega_1(x,y,\lambda_2(z,w,s))+\lambda_1(x,y,\omega_2(z,w,s))-\omega_1(y,\lambda_2(x,y,w),s)-\lambda_1(y,\omega_2(x,y,w),s)\\
		&-\omega_1(\lambda_2(x,y,z),w,s)-\lambda_1(\omega_2(x,y,z),w,s)-\omega_1(z,w,\lambda_2(x,y,s))-\lambda_1(z,w,\omega_2(x,y,s))\\
		&\omega_2(x,y,\lambda_1(z,w,s))+\lambda_2(x,y,\omega_1(z,w,s))-\omega_2(y,\lambda_1(x,y,w),s)-\lambda_2(y,\omega_1(x,y,w),s)\\
		&-\omega_2(\lambda_1(x,y,z),w,s)-\lambda_2(\omega_1(x,y,z),w,s)-\omega_2(z,w,\lambda_1(x,y,s))-\lambda_2(z,w,\omega_1(x,y,s))\\
	\end{split}
\end{equation} 
One can see that equations \eqref{equation7} and \eqref{equation8} are equivalent to 
\begin{equation}\label{equation9}
	[\Pi_1,(\mu_2,\lambda_2)]_\mathrm{LieY}+[\Pi_2,(\mu_1,\lambda_1)]_\mathrm{LieY}=0
\end{equation}
After these equations we introduce the following theorem 
\begin{thm}
	Let $(L,[\cdot,\cdot]_t^1,[\cdot,\cdot,\cdot]_t^1,[\cdot,\cdot]_t^2,[\cdot,\cdot,\cdot]_t^2)$ be an infinitesimal deformation of a compatible Lie Yamaguti algebra $(L,[\cdot,\cdot]_1,[\cdot,\cdot,\cdot]_1,[\cdot,\cdot]_2,[\cdot,\cdot,\cdot]_2)$ generated by $(\mu_1,\lambda_1,\mu_2,\lambda_2)$. \\
	Then $\big((\mu_1,\lambda_1),(\mu_2,\lambda_2) \big)\in\mathfrak{C}_c^2(L,L)$ is a (2-3)-cocycle for $(L,[\cdot,\cdot]_1,[\cdot,\cdot,\cdot]_1,[\cdot,\cdot]_2,[\cdot,\cdot,\cdot]_2)$. 
\end{thm}
\begin{proof}
Recall that if $\big((F_1,G_1),(F_2,G_2) \big)\in\mathfrak{C}_c^2(L,L)$ is a (2-3)-cocycle for $(L,[\cdot,\cdot]_1,[\cdot,\cdot,\cdot]_1,[\cdot,\cdot]_2,[\cdot,\cdot,\cdot]_2)$ means that $\delta_c((F_1,G_1),(F_2,G_2))=0$.\\
Let $(\mu_1,\lambda_1),(\mu_2,\lambda_2)\in\mathfrak{C}_c^2(L,L)$ generates an infinitesimal deformation of $(L,[\cdot,\cdot]_1,[\cdot,\cdot,\cdot]_1,[\cdot,\cdot]_2,[\cdot,\cdot,\cdot]_2)$ and by the definition of the coboundary operator $\delta_c$ we have 
\begin{equation*}
	\delta_c((\mu_1,\lambda_1),(\mu_2,\lambda_2))=\Big([\Pi_1,(\mu_1,\lambda_1)]_\mathrm{LieY},[\Pi_2,(\mu_1,\lambda_1)]_\mathrm{LieY}+[\Pi_1,(\mu_2,\lambda_2)]_\mathrm{LieY},[\Pi_2,(\mu_2,\lambda_2)]_\mathrm{LieY}\Big).
\end{equation*}	
According to equations \eqref{equation3},\eqref{equation6} and \eqref{equation9} we obtain that $\delta_c((\mu_1,\lambda_1),(\mu_2,\lambda_2))=0$.
\end{proof}
\begin{re}
	The case of compatible Lie Yamaguti algebras is more complicated than that of compatible Lie algebras. Because in the case of compatible Lie algebras if $(\mu_1,\mu_2)$ is an infinitesimal deformation of $(L,[\cdot,\cdot]_1,[\cdot,\cdot]_2)$ then $(L,\mu_1,\mu_2)$ is a compatible Lie algebras. But in the case of Lie Yamaguti algebras if $(\mu_1,\lambda_1,\mu_2,\lambda_2)$ is an infinitesimal deformation of $(L,[\cdot,\cdot]_1,[\cdot,\cdot,\cdot]_1,[\cdot,\cdot]_2,[\cdot,\cdot,\cdot]_2)$ does not mean that $(L,\mu_1,\lambda_1,\mu_2,\lambda_2)$ is a compatible Lie algebra. Only we have, after exploring the equation \eqref{eqt compatible MC2}, the two following equations which are exactly \textbf{(CLY3)} and \textbf{(CLY4)}.
	\begin{align*}
		&\bullet \lambda_1(x,y,\mu_2(z,w))-\mu_1(z,\lambda_2(x,y,w))-\mu_1(\lambda_2(x,y,z),w)\\
		&+\lambda_2(x,y,\mu_1(z,w))-\mu_2(z,\lambda_1(x,y,w))-\mu_2(\lambda_1(x,y,z),w)=0,\\
		&\bullet \lambda_1(x,y,\lambda_2(z,w,s))-\lambda_1(y,\lambda_2(x,y,w),s)-\lambda_1(\lambda_2(x,y,z),w,s)-\lambda_1(z,w,\lambda_2(x,y,s))\\
		&+\lambda_2(x,y,\lambda_1(z,w,s))-\lambda_2(y,\lambda_1(x,y,w),s)-\lambda_2(\lambda_1(x,y,z),w,s)-\lambda_2(z,w,\lambda_1(x,y,s))=0.
	\end{align*}
\end{re}

	\section{Compatible Pre-Lie Yamaguti algebra} \label{sec4}
	In this section, we study the relation between compatible pre-Lie Yamaguti algebras compatible pre-Lie Yamaguti algebras and compatible Lie Yamaguti algebras using the $ \textbf{\emph{Rota-Baxter operator}}$. From now onward, we only consider $\textbf{\emph{Rota-Baxter operator}}$ of weight $0$. \\
	We start by recalling the definition of pre-Lie algebras \cite{G1}
	\begin{defn}
		Let $L$ be a vector space over a field $\mathbb{K}$ with a bilinear product $\star: L\times L\to L$ defined by $(x, y) \mapsto x \star y$, then $L$ is called a pre-Lie algebra (or left-symmetric algebra), if 
		\begin{equation*}
			(x \star y) \star z - x \star (y \star z)=(y \star x) \star z-y \star (x \star z)
		\end{equation*}for all $x, y, z \in L$.
	\end{defn}
	It is necessary to define the compatible pre-Lie algebra.
	\begin{defn}
		A compatible pre-Lie algebra is a triple $(L, \star_1, \star_2)$ where $(L, \star_1)$ and $(L, \star_2)$ are two pre-Lie algebras such that:
		\begin{equation}
			(x \star_1 y) \star_2 z - x \star_1 (y \star_2 z)=(y \star_1 x) \star_2 z-y \star_1 (x \star_2 z).
			\label{eq1.6}
		\end{equation}
	\end{defn}
	In \cite{S1}, the authors in definition $(3.6)$ define the pre-Lie Yamaguti algebra. By using pre-Lie Yamaguti algebra and \eqref{eq1.6},  we can define the compatible pre-Lie Yamaguti algebras as follows
	\begin{defn}
		A compatible-pre Lie Yamaguti algebra (compatible pre-Lie Yamaguti algebras) is a $5$-tuple $(L, \star_1, \{\cdot, \cdot, \cdot \}_1, \star_2, \{\cdot, \cdot, \cdot \}_2)$ where $(L, \star_1, \{\cdot, \cdot, \cdot \}_1)$ and $(L, \star_2, \{\cdot, \cdot, \cdot \}_2)$ are two pre-Lie Yamaguti algebras such that :
		\begin{enumerate}
			\item[(CPLY1):] $\displaystyle\sum_{i, j \in \{1, 2\} ; i \neq j} \Big(\{z, [x, y]_{i, c}, w\}_j-\{y\star_i z, x, w\}_j+\{x\star_i z, y, w\}_j \Big)=0, $
			\item[(CPLY2):]$\{x, y, [z, w]_{2, c}\}_1+\{x, y, [z, w]_{1, c}\}_2=z \star_1 \{x, y, w\}_2+z \star_2 \{x, y, w\}_1-w \star_1\{x, y, z\}_2-w \star_2\{x, y, z\}_1, $ 
			\item[(CPLY3):] $\displaystyle\sum_{i, j \in \{1, 2\} ; i \neq j} \Big( \{\{x, y, z\}_i, w, t\}_j-\{\{x, y, w\}_i, z, t\}_j-\{x, y, \{z, w, t\}_{i, D}\}_j-\{x, y, \{z, w, t\}_{i}\}_j+\{x, y, \{w, z, t\}_{i}\}_j\\+\{z, w, \{x, y, t\}_{i}\}_{j, D} \Big)=0, $
			\item[(CPLY4):] $\displaystyle\sum_{i, j \in \{1, 2\} ; i \neq j} \Big( \{z, \{x, y, w\}_{i, D}, t\}_j+ \{z, \{x, y, w\}_i, t\}_j-\{z, \{y, x, w\}_i, t\}_j+\{z, w, \{x, y, t\}_{i, D}\}_j\\+\{z, w, \{x, y, t\}_i\}_j-\{z, w, \{y, x, t\}_i\}_j
			=\{x, y, \{z, w, t\}_i\}_{j, D}-\{\{x, y, z\}_{i, D}, w, t\}_{j}\Big)$, 
			\item[(CPLY5):]$\displaystyle\sum_{i, j \in \{1, 2\} ; i \neq j}\Big( \{x, y, z\}_{i, D} \star_j w + \{x, y, z\}_i \star_j w - \{y, x, z\}_i \star_j w = \{x, y, z \star_i w\}_{j, D}- z \star_i \{x, y, w\}_{j, D}\Big).$
		\end{enumerate} 
		where \label{Def4.3}
		\begin{align*}[x, y]_{i, c}&=x \star_i y- y \star_i x\\ \{x, y, z\}_{i, D}&= \{z, y, x\}_i- \{z, x, y\}_i+ (y, x, z)_i- (x, y, z)_i \\ (x, y, z)_i&=(x \star_i y) \star_i z - x \star_i (y \star_i z)
		\end{align*}
	\end{defn}
	\begin{lem}
		We have the following qualities: 
		\begin{enumerate}
			\item[(i)]$ 	\{[x, y]_{i, c}, z, t\}_{j, D}+\{[y, z]_{i, c}, x, t\}_{j, D}+\{[z, x]_{i, c}, y, t\}_{j, D}=0, $
			\item[(ii)] $ \{x, y, \{z, w, t\}_{i, D}\}_{j, D}-\{\{x, y, z\}_{i, D}, w, t\}_{j, D}-\{\{x, y, z\}_i, w, t\}_{j, D}+\{\{y, x, z\}_i, w, t\}_{j, D}-\{z, \{x, y, w\}_{i, D}, t\}_{j, D} \\
			-\{z, \{x, y, w\}_i, t\}_{j, D}+\{z, \{y, x, w\}_i, t\}_{j, D}-\{z, w, \{x, y, t\}_{i, D}\}_{j, D}=0.~~~~\forall i, j \in \{1, 2\}~~\textit{ such that }~i \neq j.$
		\end{enumerate}
		\label{Lem4.1}
	\end{lem}
	In the next theorem, we prove, how to define a compatible Lie Yamaguti algebras from the compatible pre-Lie Yamaguti algebras. As pre-Lie algebra give rise to Lie algebra via the commutator, the same way we give rise to compatible Lie Yamaguti algebras.
	\begin{thm}Let $(L, \star_1, \{\cdot, \cdot, \cdot\}_1, \star_2, \{\cdot, \cdot, \cdot\}_2)$ be a compatible pre-Lie Yamaguti algebra. Then the operations 
		\begin{align*}
			[x, y]_{1, c}&:=x \star_1y - y \star_1 x \\ \ [x, y]_{2, c}&:=x \star_2 y - y \star_2 x \\
			[x, y, z]_{1, c}&:=\{x, y, z\}_{1, D}+\{x, y, z\}_1-\{y, x, z\}_1 \\
			[x, y, z]_{2, c}&:=\{x, y, z\}_{2, D}+\{x, y, z\}_2-\{y, x, z\}_2
		\end{align*} define a compatible Lie Yamaguti algebra, which is called the compatible sub-adjacent Lie Yamaguti algebra and is denoted by $(L^c, [\cdot, \cdot]_{1, c}, [\cdot, \cdot, \cdot]_{1, c}, [\cdot, \cdot]_{2, c}, [\cdot, \cdot, \cdot]_{2, c})$ or simply by $L^c$.
	\end{thm}
	\begin{proof} Let us prove the compatible Lie Yamaguti algebra's identities one by one for $I=\{1, 2\}$ and $x, y, z, w, t \in L$. \begin{enumerate}\item For $(CLY1)$, consider:\begin{eqnarray*}\begin{aligned} &\displaystyle \sum_{i, j \in I;i \neq j} [[x, y]_{i, c}, z]_{j, c} + c.p.\\=&\displaystyle \sum_{i, j \in I;i \neq j} \Big( (x \star_i y) \star_j z-z \star_j (x \star_i y)-(y \star_i x) \star_j z+z \star_j (y \star_i x)+(y \star_i z) \star_j x-x \star_j (y \star_i z)\\&-(z \star_i y) \star_j x+x \star_j (z \star_i y)+(z \star_i x) \star_j y-y \star_j (z \star_i x)-(x \star_i z) \star_j y+y \star_j (x \star_i z) \Big) \\\overset{\eqref{eq1.6}}{=}& 0.
			\end{aligned}\end{eqnarray*}
			\item For $(CLY2)$, consider
			\begin{eqnarray*}
				\begin{aligned}
					&[[x, y]_{i, c}, z, t]_{j, c}+[[y, z]_{i, c}, x, t]_{j, c}+[[z, x]_{i, c}, y, t]_{j, c} \\
					=&\{x\star_i y, z, t\}_{j, D}+\{x\star_i y, z, t\}_j-\{z, x\star_i y, z, t\}_j-\{y\star_i x, z, t\}_{j, D}-\{y\star_i x, z, t\}_j+\{z, y\star_i x, t\}_j \\
					&+\{y \star_i z, x, t\}_{j, D}+\{y \star_i z, x, t\}_2-\{x, y \star_i z, t\}_2-\{z \star_i y, x, t\}_{j, D}-\{z \star_i y, x, t\}_j+\{x, z \star_i y, t\}_j \\
					&+\{z \star_i x, y, t\}_{j, D}+\{z \star_i x, y, t\}_j-\{y, z \star_i x, t\}_j -\{x \star_i z, y, t\}_{j, D}-\{x \star_i z, y, t\}_j+\{y, x \star_i z, t\}_j.
				\end{aligned}
			\end{eqnarray*}	
			Using $(CPLY1)$ in Definition \eqref{Def4.3} and  Lemma \eqref{Lem4.1}, we obtain 
			\begin{align*}
				&[[x, y]_{i, c}, z, t]_{j, c}+[[y, z]_{i, c}, x, t]_{j, c}+[[z, x]_{i, c}, y, t]_{j, c}=0.
			\end{align*}
			\item $(CLY3)$ can be proof similar to $(CLY 2)$. 
			\item For $(CLY4)$, we just use $(PCLY3)$, $(PCLY4)$ and Lemma \eqref{Lem4.1}.
			\item 
			And for $(CLY5)$, we need $(PCLY2)$ and $(PCLY5)$.
		\end{enumerate}
		So this completes the proof.
	\end{proof}
	%So we have this diagram\begin{center}$\xymatrix{\textbf{\emph{C.P.L.Y.A}}\ar[r]^{\text{commutator}} &\textbf{\emph{C.L.Y.A}} }$\end{center}
	\begin{re}
		A compatible pre-Lie Yamaguti algebras $(L, \star_c, \{\cdot, \cdot, \cdot\}_c)$ with $(\star_c=k_1. \star_{1, c}+ k_2. \star_{2, c}; ~~~\{\cdot, \cdot, \cdot\}_c= k_1. \{\cdot, \cdot, \cdot\}_{1, c}+ k_2. \{\cdot, \cdot, \cdot\}_{2, c})$ gives rise to a Lie Yamaguti algebras $(L, [\cdot, \cdot]_c, [\cdot, \cdot, \cdot]_c)$.
	\end{re} \begin{center}
		$\xymatrix{
			\text{Compatible pre-Lie Yamaguti algebras}&\underrightarrow{ commutator} &\text{Lie Yamaguti algebras} }$
	\end{center}
	To avoid confusion, we write $ \textbf{\emph{Rota-Baxter operator}}$ instead of $ \textbf{\emph{ Rota-Baxter operator }}$ of weight $\lambda=0$. Recall that a $ \textbf{\emph{Rota-Baxter operator}}$ on a Lie algebra $(L, [\cdot, \cdot])$ is a linear map satisfying:
	\begin{equation*}
		[R(x), R(y)]=R \Big([x, R(y)]+[R(x), y] \Big), \forall x, y \in L.
	\end{equation*}
	\newline
	So next, we introduce the notion of $ \textbf{\emph{Rota-Baxter operator}}$ on a compatible Lie Yamaguti algebras. Let's begin with the following definition
	\begin{defn}
		Let $(L, [\cdot, \cdot], [\cdot, \cdot, \cdot])$ be a Lie Yamaguti algebra. Let $R:L \rightarrow L$ be a linear map satisfying the following two identities
		\begin{align*}
			[R(x), R(y)]&=R \Big([x, R(y)]+[R(x), y] \Big) \\
			[R(x), R(y), R(z)]&=R \Big([R(x), R(y), z]+[R(x), y, R(z)]+[x, R(y), R(z)]\big).
		\end{align*} 
		Then, we say that $R$ is $ \textbf{\emph{Rota-Baxter operator}}$.
	\end{defn}
	\begin{re}
		If a Lie Yamaguti algebra $(L, [\cdot, \cdot], [\cdot, \cdot, \cdot])$ reduces to Lie triple system $(L, [\cdot, \cdot, \cdot])$, we obtain the notion of a $ \textbf{\emph{Rota-Baxter operator}}$ on a Lie triple system, i.e. the following identity holds
		\begin{equation*}
			[R(x), R(y), R(z)]=R \Big([R(x), R(y), z]+[R(x), y, R(z)]+[x, R(y), R(z)]\big), \forall x, y, z \in L.
		\end{equation*}
	\end{re}	
	\begin{defn}
		Let $(L, [\cdot, \cdot]_1, [\cdot, \cdot, \cdot]_1, [\cdot, \cdot]_2, [\cdot, \cdot, \cdot]_2)$ be a compatible Lie Yamaguti algebra. A linear map $R : L \rightarrow L$ is called a $ \textbf{\emph{Rota-Baxter operator}}$ on $(L, [\cdot, \cdot]_1, [\cdot, \cdot, \cdot]_1, [\cdot, \cdot]_2, [\cdot, \cdot, \cdot]_2)$ if $R$ is both a $ \textbf{\emph{Rota-Baxter operator}}$ on $(L, [\cdot, \cdot]_1, [\cdot, \cdot, \cdot]_1)$ and on $(L, [\cdot, \cdot]_2, [\cdot, \cdot, \cdot]_2)$.
	\end{defn}
	
	\begin{defn}
		A $ \textbf{\emph{Rota-Baxter}}$ compatible Lie Yamaguti algebra is a couple $(L, R)$ consisting of a compatible Lie Yamaguti algebra $(L, [\cdot, \cdot]_1, [\cdot, \cdot, \cdot]_1, [\cdot, \cdot]_2, [\cdot, \cdot, \cdot]_2)$ and $ \textbf{\emph{Rota-Baxter operator}}$  $R$ on it. It is denoted by Rota-Baxter compatible Lie Yamaguti algebras.
	\end{defn}Derivation $d$ on a compatible Lie Yamaguti algebra is defined as follows:
	\begin{ex}
		A linear map $d: L \rightarrow L$ is said to be a derivation of compatible Lie Yamaguti algebra $(L, [\cdot, \cdot]_1, [\cdot, \cdot, \cdot]_1, [\cdot, \cdot]_2, [\cdot, \cdot, \cdot]_2)$, if $d$ is both a derivation of $(L, [\cdot, \cdot]_1, [\cdot, \cdot, \cdot]_1)$ and a derivation of $(L, [\cdot, \cdot]_2, [\cdot, \cdot, \cdot]_2)$. \\Let $L$ be a compatible Lie Yamaguti algebra, given an invertible derivation $d$ on $L$, then $d^{-1}$ is a $ \textbf{\emph{Rota-Baxter operator}}$ on $L$, Thus, $(L, d^{-1})$ is $ \textbf{\emph{Rota-Baxter}}$ compatible Lie Yamaguti algebra.
	\end{ex}
	\begin{defn}
		Let $(L, [\cdot, \cdot]_1, [\cdot, \cdot, \cdot]_1, [\cdot, \cdot]_2, [\cdot, \cdot, \cdot]_2, R_L)$ and $(K, \{\cdot, \cdot\}_1, \{\cdot, \cdot, \cdot \}_1, \{\cdot, \cdot\}_2, \{\cdot, \cdot, \cdot \}_2, R_K)$ be two \textbf{\emph{R.B.C.L.Y.A}}. An homomorphism from $(L, R_L)$ to $(K, R_K)$ is a compatible Lie Yamaguti algebras homomorphism $\phi:L \rightarrow K$ such that 
		\begin{equation}
			\phi \circ R_L=R_K \circ \phi.
			\label{eq4.1}
		\end{equation}
		It is called an isomorphism if $\phi$ is an isomorphism.
	\end{defn}
	Now we introduce the notion of representation of Rota-Baxter compatible Lie Yamaguti algebras.
	\begin{defn}
		Let $(L, [\cdot, \cdot]_1, [\cdot, \cdot, \cdot]_1, [\cdot, \cdot]_2, [\cdot, \cdot, \cdot]_2, R)$ be a \textbf{\emph{Rota-Baxter}} compatible Lie Yamaguti algebra. A representation of it is a pair $(\mathbb{V}, T)$ in which $\mathbb{V}=(V, \rho, \mu)$ is a representation of compatible Lie Yamaguti algebra $L$ and $T:\mathbb{V} \rightarrow \mathbb{V}$ a linear map such that, for all $ x, y \in L \text{ and } u \in V$, we have
		\begin{align}
			\rho(R(x))(T(u))&=T\Big(\rho(R(x))u+\rho(x)(T(u)) \Big), \label{eq4.2} \\
			\mu(R(x), R(y))(T(u))&=T \Big(\mu(R(x), R(y))(u)+\mu(R(x), y)(T(u))+\mu(x, R(y))(T(u)) \Big).\label{eq4.3}
		\end{align}
	\end{defn}
	Here,  $\rho$ and $\mu$ are same as in the Definition \eqref{def3.3}.
	\begin{ex}
		Let $(L, [\cdot, \cdot]_1, [\cdot, \cdot, \cdot]_1, [\cdot, \cdot]_2, [\cdot, \cdot, \cdot]_2, R)$ be a \textbf{\emph{Rota-Baxter}} compatible Lie Yamaguti algebra, then $(\mathbb{L}, R)$ is a representation of it, where $\mathbb{L}=(L, R, R)$.
	\end{ex}
	\begin{ex}Let $(L, [\cdot, \cdot]_1, [\cdot, \cdot, \cdot]_1, [\cdot, \cdot]_2, [\cdot, \cdot, \cdot]_2, R)$ be a \textbf{\emph{Rota-Baxter}} compatible Lie Yamaguti\\ algebra, then $(\mathbb{V}, -T)$ is a representation of \textbf{\emph{Rota-Baxter}} compatible Lie Yamaguti algebra\\ $(L, [\cdot, \cdot]_1, [\cdot, \cdot, \cdot]_1, [\cdot, \cdot]_2, [\cdot, \cdot, \cdot]_2, -R)$.
	\end{ex}
	Next, we investigate, how to move from compatible Lie Yamaguti algebras to compatible pre-Lie Yamaguti algebras by using the Rota-Baxter operator.
	\begin{thm}
		Let $R:L \rightarrow L$ be a \textbf{Rota-Baxter} operator on a compatible Lie Yamaguti algebra $(L, [\cdot, \cdot]_1, [\cdot, \cdot, \cdot]_1, [\cdot, \cdot]_2, [\cdot, \cdot, \cdot]_2)$. Define the linear maps 
		\begin{equation}
			\star_i : \otimes^2L \rightarrow L \text{ such that } x \star_i y:=[Rx, y]_i, 
		\end{equation}
		\begin{equation}
			\{\cdot, \cdot, \cdot\}_i: \otimes^3L \rightarrow L \text{ such that } \{x, y, z\}_i:=[x, Ry, Rz]_i.
		\end{equation}
		for $x, y\in L$ and $i \in I=\{1, 2\}$. Then $(L, \star_1, \{\cdot, \cdot, \cdot\}_1, \star_2, \{\cdot, \cdot, \cdot\}_2)$ is a compatible pre-Lie Yamaguti algebras.
	\end{thm}
	\begin{proof}
		For $x, y, z, t, w \in L$, we have 
		\begin{equation}
			D(Rx, Ry)z:=\{x, y, z\}_D
		\end{equation}
		Let's begin with the first assertion $(C.P.L.Y.1)$
		\begin{align*}
			&\{z, [x, y]_{1, c}, w\}_2-\{y\star_1 z, x, w\}_2+\{x\star_1 z, y, w\}_2 \\
			&=[z, R[Rx, y]_1, Rw]_2-[z, R[Ry, x]_1, Rw]_2-[[Ry, z]_1, Rx, Rw]_2+[[Rx, z]_1, Ry, Rw]_2 \\
			%& \text{using the Rota-Baxter operator, we obtain} \\
			&=[z, [Rx, Ry], Rw]_2-[[Ry, z]_1, Rx, Rw]_2+[[Rx, z]_1, Ry, Rw]_2 \\
			&\overset{(C.L.Y.1)}{=}0.
		\end{align*}
		For the second identity $(C.P.L.Y.2)$, we compute
		\begin{align*}
			&\{x, y, [z, t]_{1, c}\}_2-z \star_1 \{x, y, t\}_2+t \star_1 \{x, y, z\}_2 \\
			&=[x, Ry, R[Rz, t]_1]_2-[x, Ry, R[Rt, z]_1]_2-[Rz, [x, Ry, Rt]_2]_1+[Rt, [x, Ry, Rz]_2]_1 \\
			%&\text{ Using the Rota-Baxter operator R we obtain} \\
			&=[x, Ry, [Rz, Rt]_1]_2-[Rz, [x, Ry, Rt]_2]_1+[Rt, [x, Ry, Rz]_2]_1 \\
			&\overset{(C.L.Y.5)}{=}0.
		\end{align*}
		For the third identity,  recall that
		\begin{align*}
			ad_i(x)(z)&=[x, z]_i \\
			\mathfrak{ad}_i(x, y)(z)&=[z, x, y]_i \\
			\mathfrak{L}_i(x, y)&=\mathfrak{ad}_i(y, x)-\mathfrak{ad}_i(x, y)-ad_i([x, y]_i)+[ad_i(x), ad_i(y)].\end{align*} Which leads us to define $\{x, y, z\}_{i, D}=\mathfrak{L}_i(Rx, Ry)(z)$, that obviously yields
		\begin{align*}
			&\{x, y, z\}_{i, D} \star_j w+\{x, y, z\}_{i} \star_jw-\{y, x, z\}_i \star_j w-\{x, y, z \star_i w\}_{j, D}+z \star_i \{x, y, w\}_{j, D} \\
			&=[R(\mathfrak{L}_i(Rx, Ry)(z), w)]_j+[R[x, Ry, Rz]_i, w]_j-[R[y, Rx, Rz]_i, w]_j-\mathfrak{L}_j(Rx, Ry)([Rz, w]_i)+[Rz, \mathfrak{L}_j(Rx, Ry)w]_i \\&=[[Rx, Ry, Rz]_i, w]_j-[R[Rx, Ry, z]_i, w]_j+[R(\mathfrak{L}_i(Rx, Ry)(z), w)]_j-\mathfrak{L}_j(Rx, Ry)([Rz, w]_i)+[Rz, \mathfrak{L}_j(Rx, Ry)w]_i\\ &=\mathfrak{ad}_j([Rx, Ry, Rz]_i)(w)-[\mathfrak{L}_j(Rx, Ry), ad_i(Rz)](w) \\ &=0 
		\end{align*}Hence, the identity $(C.P.L.Y.5)$ is confirmed by using  \eqref{def3.3} and  the Rota-Baxter operator on the Lie triple system. Whereas the rest of the identities are easy to verify.
	\end{proof}
	\noindent {\bf Acknowledgment:}
	The authors would like to thank the referee for valuable comments and suggestions on this article.
	%%%%%%%%%%%%%%%%%%%%%%%%%%%%%%%
	
\end{document}